\documentclass[11pt,reqno,a4paper]{amsart}
\usepackage{amsmath,amssymb}
\usepackage{cite}

\IfFileExists{hyperref.sty}{%
  \usepackage[unicode]{hyperref}
  \hypersetup{
    colorlinks=true,
    citecolor=blue,
    linkcolor=blue,
    pdfborder={0 0 0},
    pdftitle={Determining Schrodinger Potentials on Riemannian Manifolds},
    pdfauthor={Lu Chen, Yan Jiang, Hongyu Liu, and Longyue Tao},
    pdfsubject={Smooth potential uniqueness; integrated local continuation and sphere estimates},
    pdfkeywords={inverse boundary problem, smooth potential, local normalization, Green kernel}
  }
}{}
\providecommand{\href}[2]{#2}

\numberwithin{equation}{section}
\newtheorem{theorem}{Theorem}[section]
\newtheorem{lemma}{Lemma}[section]
\theoremstyle{remark}

\allowdisplaybreaks[2]
\newcommand{\opSevenStep}[1]{\smallskip\noindent\emph{#1.}\ }
\title[Determining Schr\"odinger potentials]{Determining Schr\"odinger Potentials on Riemannian Manifolds}
\author{Lu Chen}
\address[Lu Chen]{Key Laboratory of Algebraic Lie Theory and Analysis, Ministry of Education, School of Mathematics and Statistics, Beijing Institute of Technology, Beijing 100081, PR China}
\email{chenlu5818804@163.com}
\author{Yan Jiang}
\address[Yan Jiang]{Department of Mathematics, City University of Hong Kong, Hong Kong SAR, China}
\email{yjian24@cityu.edu.hk}
\author{Hongyu Liu}
\address[Hongyu Liu]{Department of Mathematics, City University of Hong Kong, Hong Kong SAR, China}
\email{hongyu.liuip@gmail.com, hongyliu@cityu.edu.hk}
\author{Longyue Tao}
\address[Longyue Tao]{Department of Mathematics, City University of Hong Kong, Hong Kong SAR, China}
\email{sdyctly@163.com, longyue.tao@my.cityu.edu.hk}
\date{October 7, 2026}
\subjclass[2020]{Primary 35R30; Secondary 58J32, 35J10, 35J25}
\keywords{Schr\"odinger potential, inverse boundary problem, local positive normalization, Green kernel, shrinking spheres}

\begin{document}
\begin{abstract}
We prove that the Dirichlet-to-Neumann map uniquely determines a smooth real Schr\"odinger potential on any prescribed compact connected smooth Riemannian manifold of dimension at least three with nonempty smooth boundary.
The proof combines local continuation of Green kernels with uniform estimates for the associated integral representations on shrinking spheres.
These estimates yield local recovery of the potential, and a continuation argument extends this recovery throughout the manifold.
\end{abstract}
\maketitle

\section{Introduction}\label{sec:introduction}
\subsection{Motivation and statement}\label{subsec:motivation}
Let $(\Omega,g)$ be a compact connected smooth Riemannian manifold of dimension $n\geq3$ with nonempty smooth boundary.
We consider the stationary Schr\"odinger operator
\[
 L_{g,q}=-\Delta_g+q,\qquad
 \Delta_g u=\mu^{-1}\partial_i(\mu g^{ij}\partial_j u),\qquad
 \mu=(\det g)^{1/2},\quad \mathrm{d}V_g=\mu\,\mathrm{d}x.
\]
Repeated indices are summed from $1$ to $n$.
Smoothness includes the boundary, and integration uses positive Riemannian densities.
For a smooth real potential $q$, suppose that
\begin{equation}\label{eq:nonresonance}
 0\notin\operatorname{Spec}_{D}(-\Delta_g+q),
\end{equation}
where $\operatorname{Spec}_{D}$ denotes the spectrum of the Dirichlet realization.
The boundary value problem
\[
 L_{g,q}u=0\quad\text{in }\Omega,\qquad
 u|_{\partial\Omega}=f
\]
then has a unique weak solution $u=u_q^f\in H^1(\Omega)$ for each $f\in H^{1/2}(\partial\Omega)$.
Its Dirichlet-to-Neumann map is
\begin{equation}\label{eq:DN}
 \Lambda_{g,q}:H^{1/2}(\partial\Omega)\longrightarrow H^{-1/2}(\partial\Omega),
 \qquad \Lambda_{g,q}f=\partial_{\nu_g}u_q^f|_{\partial\Omega},
\end{equation}
where $\nu_g$ is the outward unit normal and the weak normal derivative is paired with the known density $\mathrm{d}S_g$.
The map records the normal response on the entire boundary for every prescribed Dirichlet input.
The inverse problem asks whether the boundary map uniquely determines the potential; equivalently, whether
\begin{equation}\label{eq:question}
 \Lambda_{g,q_1}=\Lambda_{g,q_2}
 \quad\Longleftrightarrow\quad q_1=q_2\quad\text{on }\Omega.
\end{equation}

In this paper, we completely resolve this problem for smooth real potentials on arbitrary prescribed compact connected smooth Riemannian manifolds of dimension at least three with nonempty smooth boundary, using the Dirichlet-to-Neumann map.
\begin{samepage}
\begin{theorem}\label{thm:main}
Let $(\Omega,g)$ be a prescribed compact connected smooth Riemannian manifold of dimension $n\geq3$ with nonempty smooth boundary.
Let
\[
 q_1,q_2\in C^\infty(\Omega;\mathbb{R}),\qquad
 0\notin\operatorname{Spec}_{D}(-\Delta_g+q_j),\quad j=1,2.
\]
If their Dirichlet-to-Neumann maps satisfy
\begin{equation}\label{eq:main-data}
 \Lambda_{g,q_1}=\Lambda_{g,q_2},
\end{equation}
then
\begin{equation}\label{eq:main-conclusion}
 q_1=q_2\quad\text{on }\Omega.
\end{equation}
The maps are compared with the same boundary identification and the same density induced by the prescribed metric.
\end{theorem}
\end{samepage}

\subsection{Historical development}\label{subsec:history}

Potential recovery on a prescribed Riemannian background is closely connected with Calder\'on's inverse conductivity problem. In particular, within a fixed conformal class, the conductivity equation can be reduced to a Schr\"odinger equation, so that boundary measurements for the conductivity correspond to Dirichlet-to-Neumann data for an associated potential; see~\cite{calderon1980inverse,feldman2025calderon}. In Euclidean domains of dimension at least three, the complex geometric optics (CGO) method \cite{sylvester1987global} established a fundamental global uniqueness theorem. In dimension two, smooth potentials on compact Riemann surfaces can also be recovered from partial boundary data~\cite{guillarmou2011calderon}.

For higher dimensions, uniqueness results were obtained on admissible manifolds and, more generally, on conformally transversally anisotropic (CTA) manifolds under suitable assumptions on the transversal geodesic ray transform~\cite{dos2009limiting,dos2016calderon}. Uniqueness also extends to \(L^{n/2}\) potentials in admissible geometries~\cite{dos2013determining}, while constructive procedures for smooth potentials are available under appropriate geometric and ray transform assumptions~\cite{kenig2011reconstructions,feizmohammadi2021reconstruction}. The geometry plays a structural role in these arguments: limiting Carleman weights and the transversal ray transform provide the special solutions and injectivity needed to recover the potential. Geometric obstructions to the existence of such weights explain why this approach cannot simply be transferred to arbitrary smooth metrics~\cite{angulo2016obstructions}.

We also mention related developments concerning finite or partial boundary measurements~\cite{alberti2019calderon,uhlmann2026partial,daude2019non} and the recovery of unknown metrics in structured geometries~\cite{daude2021anisotropic,daude2026stability,chen2026quasianalyticity}. More recently, \cite{uhlmann2026calderon} stated a uniqueness theorem for smooth compactly supported potentials in dimension three under a \(C^4\) near-Euclidean assumption and strict boundary convexity. Another direction establishes uniqueness at sufficiently large fixed frequencies under geometric assumptions on the manifold and its geodesic ray transform~\cite{uhlmann2024anisotropic,ma2024anisotropic}. These results require frequency thresholds depending on a priori bounds for the potentials and their difference.

For the inverse problem \eqref{eq:question} on a prescribed Riemannian background in dimensions three and higher, however, global uniqueness has still remained a long-standing open challenge. In this paper, we completely resolve this problem for smooth real potentials on compact connected smooth manifolds with nonempty smooth boundary, as stated in Theorem~\ref{thm:main}. The proof uses a local continuation argument to propagate boundary information through the interior.

\subsection{The local continuation method}\label{subsec:method}
We prove~\eqref{eq:question} by extending the region on which the two potentials and their Green data agree.
The use of solution families to encode local geometry is related to the Poisson embedding approach~\cite{lassas2020poisson}.
Quantitative Runge estimates relate approximation accuracy to the size of the boundary input~\cite{ruland2019quantitative}.
Here source Runge approximation selects a fixed finite family of solutions; the uniform estimate needed below is established separately.

Equality of the boundary maps determines the boundary jet of the potential difference.
A common smooth exterior extension and a transmission argument then give equality of the exterior Green kernels.
Starting from these data, the proof addresses two main difficulties:

\begin{enumerate}
\item[(i)] \textbf{Construct and continue a mixed Green kernel.}
At a contact point, fixed remote sources generate solutions $w_j,u_j^1,\ldots,u_j^n$ such that $w_j>0$ locally and the differentials of $u_1^k/w_1$ form a basis.
The transformation
\begin{equation}\label{eq:intro-transform}
 \widehat g_j=w_j^{4/(n-2)}g,\qquad v_j^k=u_j^k/w_j
\end{equation}
makes $v_j^k$ harmonic for $\widehat g_j$ and removes the zeroth-order term.
Constants therefore solve both transformed equations.
Since both metrics are conformal to the same prescribed metric, positive scalar multiplication gives identical principal matrices in common coordinates.

We construct a mixed kernel solving the first Laplace equation in its first variable and the second in its second variable.
The initial data are known whenever either variable lies in the region of agreement.
Quantitative interpolation and a singular value expansion construct an extension from two intersecting product sets.
Local Carleman estimates and the strict product surfaces of Section~\ref{sec:product} carry the kernel to the required moving spheres.
Unique continuation ensures agreement on overlaps; the existence of the extension comes from the construction.

\item[(ii)] \textbf{Control the Green functionals as the spheres shrink.}
The flux of the continued kernel defines functionals that reproduce constants and send the selected normalized solutions for the first metric to their counterparts for the second.
Their densities need not be positive, and estimates at fixed separation from the diagonal do not control the shrinking limit.
Writing these densities as $\psi_\delta$ on a fixed local unit sphere bundle, we prove
\begin{equation}\label{eq:intro-density}
 \|\chi\psi_\delta\|_{L^2}\leq C,\qquad 0<\delta\leq1,
\end{equation}
where $\chi$ equals one near the contact point and $C$ is independent of the shrinking radius.

Differentiating the transfer representation and using the second harmonic equation yields an equation coupling center derivatives with operators on the sphere.
Constant reproduction fixes the angular mean.
A weighted estimate on the remaining spherical harmonic modes controls the complete operator, including its second derivatives in the center variable.
A cutoff confines localization errors to the known region or to fixed positive separation.
These ingredients give the uniform bound in Section~\ref{sec:density}.
\end{enumerate}

Constant reproduction expresses the difference of paired normalized solutions through the variation of the first solution over a shrinking sphere.
A first-order Taylor estimate, Cauchy--Schwarz and~\eqref{eq:intro-density} make this difference tend to zero in $L^2$ on a fixed neighborhood.
The paired functions remain fixed, so they agree there.
Their differentials form a basis, and subtracting their two equations gives $\mathrm{d}\log(w_1/w_2)=0$.
Agreement on the known side fixes the constant, and $q_j=\Delta_gw_j/w_j$ yields local equality of the potentials.
A maximal connected region argument propagates this equality through the interior.

It is pointed out that all parameters are chosen after fixing the two potentials, and the positive normalizing solutions are needed only locally. The kernel continuation and uniform density bound are established in the proof, without additional assumptions on the potentials.

\subsection{Organization of paper}
The paper follows these steps.
Section~\ref{sec:boundary} establishes common exterior Green data and the local positive normalization.
Section~\ref{sec:product} continues the mixed kernel and constructs the transfer functionals on shrinking spheres.
Section~\ref{sec:density} proves the uniform bound for their densities.
Section~\ref{sec:proof} uses this bound for local recovery and completes the global proof of Theorem~\ref{thm:main}.

\section{Boundary reduction and local normalization}\label{sec:boundary}
\subsection{Boundary determination and exterior Green data}\label{subsec:boundary-jets}
For a smooth real potential $q$ on a compact smooth manifold $N$ with boundary, we write
\[
 \mathcal B_q(u,v)=\int_N\bigl(\langle \mathrm{d}u,\mathrm{d}\overline v\rangle_g+q u\overline v\bigr)\,\mathrm{d}V_g.
\]
We use the Hermitian form for estimates and its real bilinear version for real solutions.
Standard elliptic Fredholm theory implies that nonresonance makes
$\mathcal L_q:H_0^1(N)\to H^{-1}(N)$ invertible.
The Dirichlet realization is self-adjoint with domain $H^2(N)\cap H_0^1(N)$, and the usual interior and boundary regularity apply; see~\cite{mclean2000strongly}.
These facts require no positivity of $\mathcal B_q$.
We also use open-set unique continuation for smooth real scalar elliptic equations, including smooth potentials of either sign~\cite[Section 8]{koch2005dispersive}.
It identifies existing solutions and does not by itself construct an extension.

For smooth boundary values $f,h$ and any $H^1$ extension $v$ of $h$, the weak definition of the boundary map gives
\begin{equation}\label{eq:weak-DN}
 \langle\Lambda_{g,q}f,h\rangle
 =\int_\Omega\bigl(\langle \mathrm{d}u_q^f,\mathrm{d}v\rangle_g+q u_q^f v\bigr)\,\mathrm{d}V_g.
\end{equation}
By density and boundedness, equality on smooth boundary values gives equality in the spaces of~\eqref{eq:DN}.

For a fixed smooth metric, the full Dirichlet-to-Neumann map determines the complete boundary jet of a smooth potential~\cite[Theorem 8.4]{dos2009limiting}.
This boundary determination result holds on a general compact smooth manifold; it does not require the geometric assumptions used for the interior results in that reference.
Applied with the same metric and zero magnetic terms, it gives
\begin{equation}\label{eq:flat-q}
 j^\infty_{\partial\Omega}(q_1-q_2)=0.
\end{equation}
Equality of these jets does not imply agreement on an interior neighborhood.
We use it only to paste the two potentials to a common smooth exterior extension.

The full boundary data allow us to attach a thin collar along the entire boundary.
The added region can have several components.
We shall later use only one of its connected components as a starting region.

\begin{lemma}\label{lem:cap}
Under the hypotheses of Theorem~\ref{thm:main}, there is a compact connected smooth extension $N$ of $\Omega$ with a common smooth metric and smooth potentials, still denoted $g,q_1,q_2$, such that both Dirichlet operators on $N$ are invertible.
The extended potentials are equal on the open added region $E=N^\circ\setminus\Omega$.
If $R_j=(-\Delta_g+q_j)_{D,N}^{-1}$, then
\begin{equation}\label{eq:cap-solutions}
 (R_1F)|_E=(R_2F)|_E\qquad (F\in C_c^\infty(E)).
\end{equation}
Their symmetric Green kernels satisfy
\begin{equation}\label{eq:cap-kernels}
 G_1(x,y)=G_2(x,y)\qquad (x,y\in E,\ x\ne y).
\end{equation}
\end{lemma}
\begin{proof}
\opSevenStep{1. Smooth exterior coefficients}
Extend the known metric and $q_1$ smoothly across a collar of the boundary, and use the same exterior extension for $q_2$.
The pasting is smooth by~\eqref{eq:flat-q}.
For a collar of thickness $\epsilon$, denote the extended manifold by $N_\epsilon$.
A smooth collar diffeomorphism $J_\epsilon:\Omega\to N_\epsilon$, equal to the identity away from the collar, pulls these coefficients back to ones converging smoothly to the original coefficients as $\epsilon\downarrow0$.

\opSevenStep{2. The extension remains nonresonant}
The pulled-back Dirichlet operators therefore satisfy
\[
 \|\mathcal L_{j,\epsilon}-\mathcal L_{j,0}\|_{H_0^1\to H^{-1}}
 \longrightarrow0.
\]
Since $\mathcal L_{j,0}$ is invertible, openness of the set of bounded invertible operators gives invertibility for all sufficiently small $\epsilon$.
Fix one thickness valid for both potentials and set $N=N_\epsilon$.
This choice uses nonresonance only and permits negative Dirichlet eigenvalues.

\opSevenStep{3. A weak transmission argument}
Let $F\in C_c^\infty(E)$ be real and set $w_1=R_1F$.
Its trace $h=w_1|_{\partial\Omega}$ is smooth.
Since the source is outside $\Omega$, $w_1$ is $L_{g,q_1}$-harmonic there.
Let $v_2$ solve the original $L_{g,q_2}$ Dirichlet problem on $\Omega$ with trace $h$.
Define $\widetilde w$ to equal $v_2$ in $\Omega$ and $w_1$ on $E$.

More precisely, $v_2-w_1|_\Omega\in H_0^1(\Omega)$.
Its zero extension to $N$, added to $w_1$, is an element of $H_0^1(N)$ and is exactly $\widetilde w$.
For a real $\varphi\in H_0^1(N)$, split the weak forms across the original boundary.
The exterior forms and sources coincide.
By~\eqref{eq:weak-DN}, the difference of the interior contributions is
\[
 \big\langle (\Lambda_{g,q_2}-\Lambda_{g,q_1})h,
                 \varphi|_{\partial\Omega}\big\rangle=0.
\]
Consequently
\[
 \int_N\bigl(\langle \mathrm{d}\widetilde w,\mathrm{d}\varphi\rangle_g+
                  q_2\widetilde w\varphi\bigr)\,\mathrm{d}V_g
       =\int_E F\varphi\,\mathrm{d}V_g.
\]
Thus $\widetilde w=R_2F$ by invertibility on $N$.
In particular the two solutions agree on $E$, proving~\eqref{eq:cap-solutions}.
This proof contains no minimization of an indefinite quadratic form.

\opSevenStep{4. Equality of Green kernels}
Self-adjointness of the two inverses gives symmetric distribution kernels normalized by
\begin{equation}\label{eq:Green-normalization}
 L_{g,q_j,x}G_j(x,y)=\delta_y^g(x),\qquad
 R_jF(x)=\int_NG_j(x,y)F(y)\,\mathrm{d}V_g(y).
\end{equation}
Equation~\eqref{eq:cap-solutions}, tested also in the first variable against functions supported in $E$, gives equality of these kernels on $E\times E$ as distributions.
Elliptic regularity makes them smooth away from the diagonal, proving~\eqref{eq:cap-kernels}.
The kernels need not be positive.
\end{proof}

Fix a connected component $E_0$ of the open collar $E$ and a nonempty open source region $U_0\Subset E_0$.
All sources used below belong to $C_c^\infty(U_0;\mathbb{R})$.
The extension $N$, the metric, the potentials and the inverse operators are now fixed.

\subsection{Local positive normalization}\label{subsec:source-jets}
Let $p\in N^\circ\setminus\overline{U_0}$.
Choose a small smooth coordinate ball $U$ about $p$, disjoint from $\overline{U_0}$ and with connected complement.
Source Runge approximation~\cite[Proposition A.6]{lassas2020poisson}, with Sobolev order $s>n/2+1$, transfers arbitrary local first jets to a dense set of source-generated first jets.
Local first-jet prescription is recalled in~\cite[proof of Proposition A.7]{lassas2020poisson}.
The condition $s>n/2+1$ makes evaluation of first jets continuous in the local $H^s$ topology by Sobolev embedding~\cite{adams2003sobolev,taylor2011partial}.
Since the jet space is finite-dimensional, this gives surjectivity of
\[
 f\longmapsto\bigl(R_jf(p),\mathrm{d}(R_jf)(p)\bigr),\qquad
 C_c^\infty(U_0;\mathbb{R})\longrightarrow\mathbb{R}\oplus T_p^*N.
\]
The Runge theorem assumes injectivity of the Dirichlet operator, which follows from nonresonance, and imposes no sign condition on the potential.

Suppose that $R_1f=R_2f$ on an open set $W$ with $p\in\overline W$ for every such source.
Choose real $f_0,\ldots,f_n$ for the first inverse and put $w_j=R_jf_0$, $u_j^k=R_jf_k$.
Smoothness transfers the prescribed jets to the second inverse:
\begin{equation}\label{eq:prescribed-jets}
 w_j(p)=1,\qquad \mathrm{d}w_j(p)=0,\qquad
 u_j^k(p)=0,\qquad \mathrm{d}u_j^k(p)=\mathrm{d}x^k.
\end{equation}
After shrinking to a common neighborhood $V$ disjoint from $\overline{U_0}$, we have $w_j>1/2$ and the covectors $\mathrm{d}(u_1^k/w_1)$ form a basis.
All these functions solve their respective homogeneous equations on $V$.

On the neighborhood just obtained, let
\begin{equation}\label{eq:normalized-objects}
 p_n=\frac4{n-2},\qquad
 \widehat g_j=w_j^{p_n}g,\qquad
 v_j^k=u_j^k/w_j,\qquad
 K_j(x,y)=\frac{G_j(x,y)}{w_j(x)w_j(y)}.
\end{equation}
The same original coordinate chart is used for both metrics.
We retain this point throughout the argument: changing to two separately chosen coordinate systems would not preserve the pointwise conformal relation in~\eqref{eq:normalized-objects}.

The product rule and $L_{g,q_j}w_j=0$ give
\begin{equation}\label{eq:product-rule}
 L_{g,q_j}(w_jv)
 =-w_j\bigl(\Delta_gv+2\langle \mathrm{d}\log w_j,\mathrm{d}v\rangle_g\bigr).
\end{equation}
For $\widehat g=w^{4/(n-2)}g$, direct substitution in the coordinate Laplacian gives
\begin{equation}\label{eq:conformal-Laplacian}
 \Delta_{\widehat g}v
 =w^{-4/(n-2)}\bigl(\Delta_gv+2\langle \mathrm{d}\log w,\mathrm{d}v\rangle_g\bigr).
\end{equation}
Combining these formulas,
\begin{equation}\label{eq:exact-transform}
 L_{g,q_j}(w_jv)
 =-w_j^{(n+2)/(n-2)}\Delta_{\widehat g_j}v.
\end{equation}
In particular, $\Delta_{\widehat g_j}v_j^k=0$.
The transformed equation has no zeroth-order term.
It is not an approximation to the original equation.

The conformal volume density is
\begin{equation}\label{eq:volume-transform}
 \mathrm{d}V_{\widehat g_j}=w_j^{2n/(n-2)}\mathrm{d}V_g.
\end{equation}
Together with~\eqref{eq:exact-transform} and~\eqref{eq:Green-normalization}, this gives the distributional identity
\begin{equation}\label{eq:kernel-transform}
 -\Delta_{\widehat g_j,x}K_j(x,y)=\delta_y^{\widehat g_j}(x).
\end{equation}
Symmetry gives the analogous identity in the other variable.
The usual local elliptic parametrix estimates apply to $G_j$; division by the smooth positive factors $w_j(x)w_j(y)$ preserves these bounds on compact subsets of $V$.
All Green fluxes below use the outward normal and surface density of the indicated metric.

Only this local normalization is needed.
Nonresonance permits negative Dirichlet eigenvalues and therefore does not guarantee a global positive solution.
The functions $w_j$ may vanish outside $V$, and $K_j$ is used only as a local fundamental kernel.
The source and neighborhood may be chosen anew at each contact point.

\section{Continuation and transfer of the mixed kernel}\label{sec:product}
A \emph{product solution} is a function $F(x,y)$ satisfying
\[
 L_1^xF=0,\qquad L_2^yF=0,
\]
where each $L_i$ is a smooth scalar elliptic operator with real positive-definite principal symbol on a coordinate domain.
Smooth first- and zeroth-order terms are allowed.
The sum $L_1^x+L_2^y$ is elliptic on the product.
Consequently ordinary unique continuation identifies two product solutions on a connected overlap once they agree on a nonempty open subset.
The lemmas below address the separate issue of constructing the extension.

\subsection{Extension from two intersecting product sets}
\begin{lemma}\label{lem:cross}
Let $D_i$ be bounded connected coordinate domains and let $S_i\Subset D_i$ be nonempty open sets.
Suppose a product solution is given consistently on
\[
 (D_1\times S_2)\cup(S_1\times D_2),
\]
with the sum of its two $L^2$ norms at most $M$.
Suppose $O_i\Subset D_i$ are open sets and every $L_i$-solution $w$ on $D_i$ satisfies
\begin{equation}\label{eq:cross-interpolation}
 \|w\|_{L^2(O_i)}\leq C\|w\|_{L^2(S_i)}^{\gamma_i}
                         \|w\|_{L^2(D_i)}^{1-\gamma_i},
 \qquad 0<\gamma_i<1,\quad\gamma_1+\gamma_2>1.
\end{equation}
Then the compatible function extends as a product solution to $O_1\times O_2$.
Its $C^m$ norm on every smaller compact subset is bounded by $C_mM$.
The constants are uniform when the coefficient bounds, compact containments, interpolation constants and the positive margin $\gamma_1+\gamma_2-1$ are uniform.
\end{lemma}
\begin{proof}
We may replace $O_i$ by $O_i\cup S_i$.
Indeed, if $a=\|w\|_{L^2(S_i)}$ and $b=\|w\|_{L^2(D_i)}$, then $a\leq b$ and $a\leq a^{\gamma_i}b^{1-\gamma_i}$.
The interpolation inequality remains valid on the union with a larger constant.
Both comparison sets used below therefore belong to the region of convergence.
Let $\mathcal H_1(D_1)$ be the closed subspace of $L^2(D_1)$ consisting of distributional $L_1$-solutions.
Restriction
\[
 R:\mathcal H_1(D_1)\longrightarrow L^2(S_1)
\]
is compact by interior elliptic estimates and injective by unique continuation.
The compact positive operator $R^*R$ has an orthonormal eigenbasis $(w_j)$ in $\mathcal H_1(D_1)$ with
\begin{equation}\label{eq:cross-singular}
 (w_i,w_j)_{L^2(S_1)}=\mu_j^2\delta_{ij},\qquad 0<\mu_j\leq1.
\end{equation}
Every fixed power of $j$ is an upper decay bound for these singular values: for each integer $N$,
\begin{equation}\label{eq:cross-decay}
 \mu_j\leq C_N(1+j)^{-N}.
\end{equation}
Indeed, multiply a solution by a fixed smooth cutoff supported in $D_1$ and equal to one near $S_1$.
Its $H^k$ norm is bounded by its $L^2(D_1)$ norm.
Fourier truncation in a containing cube gives a rank $O(A^n)$ approximation with error $O(A^{-k})$.
The minimax characterization of singular values gives $\mu_j\leq C_kj^{-k/n}$.
Taking $k$ sufficiently large proves~\eqref{eq:cross-decay}.
This argument uses interior estimates only.

For $y\in S_2$, expand the first given piece as
\[
 F(x,y)=\sum_jw_j(x)b_j(y),\qquad \|b_j\|_{L^2(S_2)}\leq M.
\]
Its coefficients extend to $D_2$ by
\begin{equation}\label{eq:cross-coefficients}
 b_j(y)=\mu_j^{-2}\int_{S_1}F(x,y)\overline{w_j(x)}\,\mathrm{d}x,
 \qquad \|b_j\|_{L^2(D_2)}\leq M\mu_j^{-1}.
\end{equation}
They solve $L_2b_j=0$, and~\eqref{eq:cross-singular} shows agreement with the original coefficients on $S_2$.
Applying~\eqref{eq:cross-interpolation} yields
\[
 \|w_j\|_{L^2(O_1)}\leq C\mu_j^{\gamma_1},\qquad
 \|b_j\|_{L^2(O_2)}\leq CM\mu_j^{\gamma_2-1}.
\]
Thus the series converges absolutely in $L^2(O_1\times O_2)$, because its summands are bounded by $CM\mu_j^{\gamma_1+\gamma_2-1}$ and~\eqref{eq:cross-decay} makes their sum finite.
More explicitly, put $\eta=\gamma_1+\gamma_2-1>0$ and choose $N$ with $N\eta>1$.
Then
\[
 \sum_j\|w_j(x)b_j(y)\|_{L^2(O_1\times O_2)}
 \leq CM C_N^\eta\sum_j(1+j)^{-N\eta}<\infty.
\]
For a uniform margin $\eta\geq\eta_0>0$, choose one $N$ with $N\eta_0>1$.
The slice equalities initially hold almost everywhere by Fubini's theorem; after passage to the limit, elliptic regularity makes these distributional equalities pointwise on the overlaps.
The limit solves both equations distributionally.
Interior elliptic estimates for their sum give smoothness and the stated compact-subset bounds.

We verify agreement with the second piece as well.
If $z\in L^2(S_1)$ is orthogonal to the closure of $R\mathcal H_1(D_1)$, then
\[
 y\longmapsto\int_{S_1}F(x,y)\overline{z(x)}\,\mathrm{d}x
\]
solves $L_2w=0$ on $D_2$ and vanishes on $S_2$.
Hence it vanishes on $D_2$.
Every second-piece slice therefore belongs to that closed range.
Its expansion in the orthonormal basis $(w_j/\mu_j)$ has coefficients $\mu_jb_j(y)$ by~\eqref{eq:cross-coefficients}.
The same series is consequently the given second piece on $S_1\times O_2$.
It agrees with the first piece on $O_1\times S_2$ by its construction.
No positive lower bound for $\mu_j$ has been used, so all estimates have the claimed uniformity.
\end{proof}

\subsection{Continuation across product surfaces}\label{subsec:product-surfaces}
For a positive principal metric $g$, a smooth real function $\theta$ with nonzero differential is a strict elliptic weight at a point if
\begin{equation}\label{eq:strict-weight}
 \operatorname{Hess}_g\theta(e,e)+\operatorname{Hess}_g\theta(v,v)>0,
 \quad v\perp e,\quad |v|_g=1,\quad e=\nabla^g\theta/|\nabla^g\theta|_g.
\end{equation}
The local Carleman estimate for such a weight, after decreasing the coordinate neighborhood, is
\begin{equation}\label{eq:strict-Carleman}
 \tau^3\|\mathrm{e}^{\tau\theta}w\|_2^2+	au\|\nabla(\mathrm{e}^{\tau\theta}w)\|_2^2
 \leq C\|\mathrm{e}^{\tau\theta}Lw\|_2^2,
 \qquad \tau\geq\tau_0,
\end{equation}
for compactly supported $w$.
This is the standard estimate for a strongly pseudoconvex elliptic weight~\cite[Section 8, Theorem 9(a)]{koch2005dispersive}; condition~\eqref{eq:strict-weight} is its Riemannian form for the Laplace principal symbol.
Smooth lower-order terms are absorbed for sufficiently large $\tau$.
The constants can be chosen uniformly for bounded coefficient families with a fixed positive pseudoconvexity margin.

\begin{lemma}\label{lem:strict-product}
Let $\rho$ be smooth near $z_0=(x_0,y_0)$ with $\rho(z_0)=0$ and $\mathrm{d}_x\rho,\mathrm{d}_y\rho\ne0$.
The norms and Hessians below use the two principal metrics.
Define
\[
 a_i=\frac{\nabla_i\rho}{|\nabla_i\rho|^2},\qquad H=\operatorname{Hess}\rho.
\]
Suppose at $z_0$ that
\begin{equation}\label{eq:strict-product-condition}
 H((a_1,-a_2),(a_1,-a_2))+H((z_1,z_2),(z_1,z_2))>0
\end{equation}
whenever $z_i\perp a_i$ and $|z_i|=|a_i|$.
Every product solution given on the side $\{\rho>0\}$ near $z_0$ extends to a neighborhood of $z_0$.
The construction uses data on two compact product sets strictly inside that side.
The output neighborhoods and their compact-subset bounds are uniform under small smooth perturbations preserving the strict inequalities and the chosen data bounds.
\end{lemma}
\begin{proof}
\opSevenStep{1. Splitting the Hessian between the factors}
We find real symmetric forms $P_i$ and $K>0$ such that
\begin{equation}\label{eq:Hessian-splitting}
 P_i(a_i,a_i)+P_i(z_i,z_i)>0,
 \qquad H+K\,\mathrm{d}\rho\otimes \mathrm{d}\rho>P_1\oplus P_2.
\end{equation}
Here and below a strict matrix inequality means positive definiteness.
We give the finite-dimensional separation argument.
The dual tests for the open cone consisting of positive definite forms plus the admissible $P_1\oplus P_2$ are positive semidefinite matrices $D$ whose diagonal blocks have the form
\[
 D_{ii}=d_i a_i\otimes a_i+S_i,
 \quad d_i\geq0,\quad S_i\geq0\text{ on }a_i^\perp,
 \quad\operatorname{tr}S_i=d_i|a_i|^2.
\]
Indeed, the dual cone in each block is generated by $a_i\otimes a_i+z_i\otimes z_i$ with $z_i$ as in the hypothesis.
Normalize $\operatorname{tr}D=1$.

It suffices to show $H:D>0$ on the normalized tests with $D\,\mathrm{d}\rho=0$.
The block equations imply $d_1=d_2=d$ and that every mixed normal--tangential block vanishes.
To verify this directly, write $\ell_1=\mathrm{d}_x\rho$ and $\ell_2=\mathrm{d}_y\rho$ for the partial covectors, so $\ell_i(a_i)=1$ and $\ell_i$ vanishes on $a_i^\perp$.
Then
$D_{11}\ell_1=d_1a_1$, $D_{22}\ell_2=d_2a_2$, and $D\,\mathrm{d}\rho=0$ forces $D_{12}\ell_2=-d_1a_1$ and $D_{21}\ell_1=-d_2a_2$.
Pairing these identities with $\ell_1,\ell_2$ and using symmetry gives $d_1=d_2$.
They also give the claimed vanishing of all normal--tangential components.
Thus
\[
 D=d(a_1,-a_2)\otimes(a_1,-a_2)+D_\perp,
 \qquad \operatorname{tr}(D_\perp)_{ii}=d|a_i|^2.
\]
Here $D_\perp\geq0$ acts on the two tangential spaces.
The value $d=0$ would force $D=0$, which is impossible after normalization.

Among positive semidefinite tangential matrices with these two prescribed traces, every extreme point has rank one.
In fact, on an image of rank $r\geq2$, symmetric perturbations have dimension $r(r+1)/2>2$.
One can therefore preserve both traces with a nonzero symmetric perturbation; sufficiently small perturbations of either sign preserve positivity.
This excludes such a matrix from being extreme.
The rank-one matrices are $d(z_1,z_2)\otimes(z_1,z_2)$ with $|z_i|=|a_i|$.
The hypothesis gives $H:D>0$ on these extreme tests, and hence on their compact convex hull.
Compactness gives a positive minimum for the normalized tests with $D\,\mathrm{d}\rho=0$.
For all other tests, increasing $K$ makes $H:D+K\,D(\mathrm{d}\rho,\mathrm{d}\rho)$ positive uniformly; a compactness argument proves this assertion by taking a subsequence if it failed.
Convex separation now yields~\eqref{eq:Hessian-splitting}.

Prescribe the gradients $\mathrm{d}_i\rho$ and the Hessians $P_i$ of functions $\theta_i$ vanishing at the two respective points.
These gradients are nonzero and~\eqref{eq:Hessian-splitting} makes each function a strict weight after shrinking the neighborhoods.
Taylor's formula gives
\begin{equation}\label{eq:weight-gap}
 \theta_1(x)+\theta_2(y)
 \leq\widetilde\rho(x,y)-c(|x-x_0|^2+|y-y_0|^2),
 \qquad \widetilde\rho=\rho+K\rho^2/2.
\end{equation}
The positive side of $\widetilde\rho$ is that of $\rho$ near $z_0$.

\opSevenStep{2. The cross sets and the interpolation exponents}
Use local coordinates $(s_i,z_i)$ with $s_i=\theta_i$.
Choose $\kappa>0$ sufficiently small that $\varphi(s_i)=s_i-\kappa s_i^2$ is still a strict weight.
Then choose successively a small lateral radius $R$, a height $\ell\ll cR^2$, and $e\ll\kappa\ell^2$.
Set
\[
 D_i=\{-\ell<s_i<\ell+4e,\ |z_i|<R\}.
\]
Choose a cutoff supported in this cylinder, equal to one in the inner lateral half and for $-\ell+3e<s_i<\ell+2e$, with bottom transition in $(-\ell+2e,-\ell+3e)$ and top transition in $(\ell+2e,\ell+3e)$.
Let $S_i\Subset D_i$ contain slightly enlarged neighborhoods of its top and lateral derivative supports, and an inner slice near $s_i=\ell+5e/4$.

The products $D_1\times S_2$ and $S_1\times D_2$ lie strictly on the given side.
On a top slice the sum of the two $s$ coordinates is positive.
On a lateral slice this sum is greater than $-2\ell$, while the quadratic gap in~\eqref{eq:weight-gap} is bounded below by a fixed multiple of $R^2$.
The choices therefore give a positive margin.
By keeping the cutoff supports compact and, if necessary, slightly reducing the cylinder end faces, both data sets have closures in $\{\rho>\epsilon\}$ for one $\epsilon>0$.

Apply~\eqref{eq:strict-Carleman} to the cutoff times an $L_i$-solution.
Normalize its $L^2(D_i)$ norm to one.
Interior estimates bound the bottom derivative terms by that norm with weight $\varphi(-\ell+3e)$.
The other derivative terms are bounded by its $L^2(S_i)$ norm with weight at most $\varphi(\ell+3e)$.
On a smaller inner cylinder starting at $s_i=-e$, the weight is at least $\varphi(-e)$.
Absorbing the polynomial factors for $\tau\geq\tau_0$ gives
\[
 \|w\|_{L^2(O_i)}
 \leq C\left(\mathrm{e}^{-a\tau}\|w\|_{L^2(D_i)}
              +\mathrm{e}^{b\tau}\|w\|_{L^2(S_i)}\right),
\]
where $a=\varphi(-e)-\varphi(-\ell+3e)>0$ and $b=\varphi(\ell+3e)-\varphi(-e)>0$.
Balancing the exponentials proves~\eqref{eq:cross-interpolation} with
\begin{equation}\label{eq:cross-exponent}
 \gamma_i=\frac{\varphi(-e)-\varphi(-\ell+3e)}
                  {\varphi(\ell+3e)-\varphi(-\ell+3e)}>\frac12.
\end{equation}
At $e=0$ this quotient is $1/2+\kappa\ell/2$; the choice $e\ll\kappa\ell^2$ preserves the strict inequality.
An explicit margin follows from
\[
 a+b=2\ell(1-6\kappa e),\qquad
 a-b=2\kappa\ell^2-8e+16\kappa e^2.
\]
Choosing $\kappa\ell<1/8$ and $0<e<\kappa\ell^2/16$ gives $a,b>0$ and $\gamma_i=a/(a+b)>1/2$.
When the small norm is positive, the balancing value is
\[
 \tau=(a+b)^{-1}\log\frac{\|w\|_{L^2(D_i)}}{\|w\|_{L^2(S_i)}}.
\]
If the small norm is zero, unique continuation gives $w=0$.
If the balancing parameter is below $\tau_0$, the same estimate follows by increasing $C$ and using the normalized large norm.
The smaller output cylinder may include a connected inner tube from the origin to the top slice.

\opSevenStep{3. Construction and agreement}
Lemma~\ref{lem:cross} now produces a product solution near $z_0$.
It agrees with the old solution on the nearby positive side.
Indeed, from a sufficiently close positive-side point, increase $s_2$ into its top slice.
The partial derivative $\partial_{s_2}\rho$ is positive after decreasing the neighborhood, so the segment remains in the given side and in the larger output cylinder.
Agreement on the slice and unique continuation along this segment prove agreement at the original point.
Every choice has a strict geometric margin.
Keeping the finite compact sets and slightly reducing the output neighborhoods preserves all estimates under the perturbations stated in the lemma.

The proof is complete.
\end{proof}

We next give the gluing argument for a compact family.
Suppose product coordinates $(z,a)$ are defined near a compact cylinder $B\times[a_-,a_+]$.
A product solution is already defined near the top and lateral boundary and on an open subset $S$ that is upward closed in $a$ at each fixed $z$.
Suppose each level $a=\text{constant}$ satisfies Lemma~\ref{lem:strict-product} outside $S$, with its positive side toward increasing $a$.
Then the solution continues through the whole cylinder, consistently with $S$.

Here is the full gluing argument.
Starting at the top, take the infimum $a_*$ of levels reached by a consistent continuation.
At a point of the level $a=a_*$ outside $S$, Lemma~\ref{lem:strict-product} uses two compact sets strictly above that level.
A finite cover of the compact level puts all these sets above $a_*+\eta$ for one $\eta>0$, where a preceding extension is available.
Reduce the output neighborhoods to boxes $B_i\times(a_*-d_i,a_*+d_i)$.
Every connected component of an overlap has a vertical segment to a height greater than $a_*$.
Both new solutions agree with the preceding solution there.
Unique continuation for $L_1^x+L_2^y$ makes them equal on that component.
The same argument gives agreement with $S$, because upward segments starting in $S$ stay there.
The finite cover has a positive minimum width, so it passes $a_*$, unless it has already included the bottom endpoint.
This proves the assertion.
The same proof gives uniform compact-subset bounds for a fixed family with uniform strict margins.
It uses data strictly above each level and never assumes a norm bound at an unreached limiting surface.
For uniformity, perform the reference construction on a slightly larger compact cylinder.
Finitely many output neighborhoods cover the target cylinder, and each is obtained from a finite sequence of local continuation steps.
Retain those finitely many sequences, with their compact input sets and their order of construction.
All geometric and interpolation margins are then fixed and positive.
Small perturbations preserve these choices, and composing the finitely many local bounds gives a uniform estimate.
This statement concerns a fixed compact family and gives no bound as its separation from the diagonal tends to zero.

\subsection{Continuation at a fixed separation}\label{subsec:fixed-separation}
We spell out how the initial kernel reaches the fixed positive distances needed later.
Let two smooth local metrics $g_1,g_2$ agree on a set $W$ containing the lower side of a smooth hypersurface through $0$.
Normalize $g_j(0)=I$ by one common linear coordinate change, and write that side as
\[
 x_n<q(x'),\qquad q(0)=0,\quad \mathrm{d}q(0)=0.
\]
Suppose symmetric local fundamental kernels $K_j$ have their metric point-mass normalization, ordinary local Green bounds, and agree on $W\times W$ off the diagonal.
Define
\[
 g_{j,\lambda}(z)=g_j(\lambda z),\qquad
 K_{j,\lambda}(x,y)=\lambda^{n-2}K_j(\lambda x,\lambda y),
 \qquad W_\lambda=\lambda^{-1}W.
\]
They supply the compatible mixed data
\begin{equation}\label{eq:product-initial-kernel}
 K_\lambda(x,y)=
 \begin{cases}K_{1,\lambda}(x,y),&y\in W_\lambda,\\
 K_{2,\lambda}(x,y),&x\in W_\lambda,
 \end{cases}\qquad x\ne y.
\end{equation}
The kernels may be local fundamental kernels; a remote Dirichlet boundary is not used in this subsection.

\begin{lemma}\label{lem:fixed-separation}
Fix $B,d,\eta>0$.
For all sufficiently small $\lambda>0$, the data~\eqref{eq:product-initial-kernel} continue as a product solution to a neighborhood of
\[
 \{(x,y):|x|,|y|\leq B,\ |x-y|\geq d\}.
\]
Every fixed $C^m$ norm on a smaller neighborhood is bounded uniformly in $\lambda$.
The continuation agrees with each prescribed piece on its overlap.
All geometric sets and positive separation margins are fixed before $\lambda$ is chosen.
\end{lemma}
\begin{proof}
\opSevenStep{1. Crossing the limiting plane}
Fix $x_0$ of height zero and a positive separation scale $s$.
In the second factor use an annulus
\[
 D_2=\{s/2<|y-x_0|<2s\}
\]
and an observed ball $S_2$ about $x_0-se_n$.
In the first factor use a ball of radius $3\tau_*s$ centered at $x_0-\tau_*se_n$, where $\tau_*>0$ is small and fixed.
Its observed concentric ball has radius $(1-\alpha)\tau_*s$.
These two large domains are separated, and the observed balls lie strictly below the plane.

There is an interpolation estimate in the second factor from $S_2$ to a slightly enlarged annulus $4s/5\leq|y-x_0|\leq6s/5$, with an exponent $\gamma_2>0$ independent of $\alpha$.
To see this, the Euclidean weight $|z-z_*|^{-1}$ has radial Hessian $2|z-z_*|^{-3}$ and tangential Hessian $-|z-z_*|^{-3}$, and is therefore strict on an annulus.
Apply~\eqref{eq:strict-Carleman} with radial cutoffs and balance the two exponentials.
A finite chain of enlarged balls contained in $D_2$ joins $S_2$ to the compact target annulus.
Multiplying the finitely many positive exponents gives $\gamma_2>0$.
The annulus is connected and the same fixed weights work for small smooth perturbations of the Euclidean coefficients.

In the first factor keep the outer cutoff a fixed distance beyond radius $\tau_*s$, and take an inner cutoff just inside radius $(1-\alpha)\tau_*s$.
Take the target ball to have radius $(\tau_*+c)s$.
As $\alpha,c\downarrow0$, the target weight approaches the observed weight, and the corresponding exponent $\gamma_1$ tends to one.
Choose them so that $\gamma_1+\gamma_2>1$.
Lemma~\ref{lem:cross} gives a continuation to a small ball around $x_0$ times the target annulus.
The rescaled original kernels have uniform derivative bounds on all these separated sets.
This follows from their local singularity of order $|x-y|^{2-n}$ and interior estimates after rescaling.

The constructed function agrees with the two original lower pieces by one-variable unique continuation on the connected comparison balls and annuli.
Retain smaller first balls and narrower annuli for coverage, while keeping the larger products for comparisons.
Two overlapping retained products of the same orientation can be compared by lowering the near-plane variable into the known region within their larger first balls.
For opposite orientations both variables are near the plane; lower both by
\[
 h=\max(x_n,y_n,0)+2\eta_0.
\]
Choose the retained widths small relative to the larger ball and annular margins, and then choose $\eta_0>0$ smaller than the minimum retained width on the fixed compact set of scales and locations.
These paths stay in the two larger products.
Their endpoints have both heights below $-\eta_0$, where the functions coincide.
Unique continuation on neighborhoods of the paths proves agreement on each overlap component.
Finally choose $\lambda$ so that the required lower compact sets lie in $W_\lambda$.

Interchanging the factors gives, for a fixed $c_0>0$, a consistent initial continuation on
\begin{equation}\label{eq:coarse-region}
 \min(x_n,y_n)<c_0|x-y|
\end{equation}
on every prescribed bounded range of positive separations.
For coverage near the plane, center the first ball at $(x',0)$ and choose $s=|y-(x',0)|$ when $x$ has the smaller nonnegative height.
The retained products contain the pair once $c_0$ is sufficiently small.
Points already below the plane belong to the original cross or to the same construction near it.

\opSevenStep{2. Sweeping variable-radius tubes}
Choose $M>B+3$ and write
\[
 t=M+y_n+e_0|y'|^2,\qquad r_a(y)=a\,t^{1+\sigma},\qquad
 x=y+r_a(y)\omega,\quad |\omega|=1.
\]
The base is $y_n\geq-B-1$, $t\leq T_0$, together with the unit directions.
Choose $a_{\min}>0$ so small that every target distance exceeds $r_{a_{\min}}(y)$ on $|y|\leq B$, and choose $a_{\max}$ so large that its entire level lies in~\eqref{eq:coarse-region}.
Choose $T_0$ so that
$c_0a_{\min}T_0^\sigma>1$.
Then the lower and upper faces of the base and the initial $a_{\max}$ level belong to the known region.
The small number $e_0>0$, chosen last, makes the base compact.
The known set is upward closed in $a$, since~\eqref{eq:coarse-region} is equivalent to
\[
 y_n<(c_0-\min(\omega_n,0))r_a(y)
\]
and the coefficient of $r_a$ is positive.

We verify the strict criterion at the Euclidean limit.
Let
\[
 \rho(x,y)=\tfrac12(|x-y|^2-r_a(y)^2),\qquad b=\nabla r_a.
\]
After a common positive normalization, the endpoint vectors in~\eqref{eq:strict-product-condition} have the form
\[
 U_1=\omega+\mathrm{i} v,\quad v\perp\omega,\ |v|=1,
 \qquad U_2=\frac{\omega+b}{|\omega+b|^2}+\mathrm{i} w,
 \quad w\perp\omega+b,\ |w|=|\omega+b|^{-1}.
\]
A real Hessian evaluated on such complex vectors means the sum of its values on the real and imaginary parts.
The identity
$\omega\cdot(U_1-U_2)=b\cdot U_2$
cancels the normal part of the distance Hessian.
The strict expression is therefore
\begin{equation}\label{eq:tube-test}
 |\pi_{\omega^\perp}(U_1-U_2)|^2
             -(r_a\partial^2r_a)(U_2,\overline{U_2}).
\end{equation}
For bounded $b$, outside small fixed neighborhoods of $b=-\omega$ and $b=-2\omega$, its first term is at least $c|b|^2$.
Indeed, a zero makes the real part of $U_2$ parallel to $\omega$ and forces equality of the projected imaginary lengths.
These two equalities imply $b=0$ or $b=-2\omega$.
Near zero the real projected difference controls $\pi_{\omega^\perp}b$ to first order, and
$|\pi_{\omega^\perp}w|=1-\omega\cdot b+O(|b|^2)$
controls the remaining component.
Compactness gives the bound away from zero and the two excluded neighborhoods.

Outside the known region, $y_n\geq c_0r_a$, hence $r_a/t\leq1/c_0$ and
\[
 b=(1+\sigma)(r_a/t)\nabla t
\]
is bounded.
The excluded neighborhoods already lie in the known region when $\sigma$ and the gradient error are small: there
\[
 \frac{x_n}{r_a}\leq\frac{(1+\sigma)|\nabla t|}{|b|}+\omega_n,
\]
which is as small as required near $b=-\omega$ and negative near $b=-2\omega$.
On the complement $|U_2|$ is bounded.
Differentiating $r_a=a t^{1+\sigma}$ gives
\[
 r_a\partial^2r_a
 =(1+\sigma)\sigma\frac{r_a^2}{t^2}\,\mathrm{d}t\otimes \mathrm{d}t
 +(1+\sigma)\frac{r_a^2}{t}\,\partial^2t,
\]
so
\begin{equation}\label{eq:tube-Hessian-bound}
 r_a|\partial^2r_a|\leq C(\sigma+T_0e_0)|b|^2.
\end{equation}
Choose the excluded neighborhoods in terms of $c_0$, then choose $\sigma>0$ small, the $a$ interval and $T_0$, and finally $e_0$ so small that~\eqref{eq:tube-Hessian-bound} is absorbed in~\eqref{eq:tube-test}.
This also makes $|\nabla t-e_n|\leq2\sqrt{e_0T_0}$ small.
Both partial gradients are nonzero where continuation is required, and the strict expression is positive there.

The resulting cylinder is compact, bounded in coordinates, and separated from the diagonal by at least
$a_{\min}(M-B-1)^{1+\sigma}>0$.
Small smooth perturbations preserve the strict inequalities.
Choose $\lambda$ accordingly and apply the compact-family argument above.
It reaches every target pair, since either its parameter is in the chosen interval or it already belongs to the known region.
The estimates are uniform for these fixed sets.

\opSevenStep{3. Agreement with the whole prescribed cross}
Perform the construction with a larger bound $B'>B$ and a smaller $0<d'<d$.
They can be chosen so that
\[
 D_y=B(0,B')\setminus\overline{B(y,d')}
\]
is connected for every $|y|\leq B$, the removed ball is strictly inside the large ball, and $D_y$ contains one fixed lower ball with height below $-\eta$.
For $y\in W_\lambda$, the new kernel and $K_{1,\lambda}$ solve the same first-variable equation on $D_y$.
They coincide on its lower ball, where both variables lie in the known set.
Unique continuation proves agreement on the target overlap.
Interchange the variables for the other piece.
The unused margins give a neighborhood of the stated compact set and the asserted compact-subset estimates.

The proof is complete.
\end{proof}

The constants in Lemma~\ref{lem:fixed-separation} may deteriorate when $d\downarrow0$.
The lemma is used only at fixed positive distances.
A separate estimate, proved in Section~\ref{sec:density}, controls the transfer density when the spheres actually shrink.

\subsection{Geometry of the shrinking spheres}\label{subsec:sphere-geometry}
We now apply the continuation results to a local pair of metrics $\widehat g_1,\widehat g_2$ arising from~\eqref{eq:normalized-objects}.
More generally, the arguments apply to any pointwise conformal smooth pair with the following local data.
A common coordinate neighborhood $V$ contains a contact point $0$ and a set $W$ containing one side of a smooth hypersurface through $0$.
The two metrics agree on $W$.
Symmetric kernels $K_j$ satisfy~\eqref{eq:kernel-transform}, the usual local kernel estimates, and
\begin{equation}\label{eq:local-known-data}
 K_1(x,y)=K_2(x,y)\quad(x,y\in W,\ x\ne y).
\end{equation}
There are finitely many harmonic pairs $(v_1^a,v_2^a)$, defined on $V$ and equal on $W$.
The original coordinate functions themselves need not be harmonic.

Off the diagonal, define
\begin{equation}\label{eq:initial-cross}
 K(x,y)=
 \begin{cases}
 K_1(x,y),&y\in W,\\
 K_2(x,y),&x\in W.
 \end{cases}
\end{equation}
The prescriptions agree where both variables lie in $W$.
Each solves the first transformed equation in $x$ and the second in $y$: in the restricted variable the two metrics agree.
The preceding continuation lemmas therefore apply.
A local fundamental-solution normalization and separated smooth bounds suffice for their construction; no condition at a remote artificial boundary is imposed.

Make one common linear coordinate change so that
\[
 \widehat g_1(0)=\widehat g_2(0)=I,
\]
and so that the known side is locally below a graph tangent to $x_n=0$.
For a positive spatial parameter $\lambda$, define
\begin{equation}\label{eq:spatial-dilation}
 \widehat g_{j,\lambda}(x)=\widehat g_j(\lambda x),\qquad
 K_{j,\lambda}(x,y)=\lambda^{n-2}K_j(\lambda x,\lambda y),\qquad
 v_{j,\lambda}^a(x)=v_j^a(\lambda x).
\end{equation}
The kernel normalization is unchanged in the rescaled metric.
On each fixed bounded coordinate set the metric coefficients converge smoothly to $I$.
The rescaled known set contains every fixed compact subset of the lower half-space once $\lambda$ is small enough.
For each fixed positive separation, all required derivatives of $K_{j,\lambda}$ are uniformly bounded by the local kernel estimates.
Lemma~\ref{lem:fixed-separation} therefore applies to~\eqref{eq:initial-cross} with these local kernels.
In particular, the construction preserves the actual matched Green data instead of replacing them by unrelated kernels on new balls.

On a neighborhood of the fixed cylinder
\[
 Y=\{y=(y',y_n):|y'|\leq1,\ -1\leq y_n\leq1\},
 \qquad t_0(y)=y_n+|y'|^2,
\]
choose a large fixed $L$ and define
\begin{equation}\label{eq:trace-normalization}
 \gamma=\mathrm{e}^{-2Lt_0}\widehat g_{1,\lambda},\qquad
 Q=\frac{n\widehat g_{2,\lambda}^{-1}}
              {\operatorname{tr}(\gamma\widehat g_{2,\lambda}^{-1})},
 \qquad H=Q-\gamma^{-1}.
\end{equation}
The first equation can be multiplied by $\mathrm{e}^{2Lt_0}$ to have principal matrix $\gamma^{-1}$.
The second can be multiplied by the positive scalar in the definition of $Q$.
These operations change neither set of solutions and introduce no zeroth-order term.

Pointwise conformality gives a positive function $\alpha_\lambda$ with
\[
 \widehat g_{2,\lambda}^{-1}
 =\alpha_\lambda\widehat g_{1,\lambda}^{-1}.
\]
Taking the trace in~\eqref{eq:trace-normalization} gives
\[
 \operatorname{tr}(\gamma\widehat g_{2,\lambda}^{-1})
 =n\alpha_\lambda \mathrm{e}^{-2Lt_0}.
\]
Therefore
\begin{equation}\label{eq:H-zero}
 Q=\mathrm{e}^{2Lt_0}\widehat g_{1,\lambda}^{-1}=\gamma^{-1},
 \qquad H\equiv0.
\end{equation}
Every derivative of $H$ vanishes.
Thus the strict surface calculation and the operator estimate below have identical normalized principal matrices at every positive radius.
There is no radius-dependent tensor difference to control.

Define the depth function
\begin{equation}\label{eq:depth}
 t(y)=\int_{1/8}^{t_0(y)}\mathrm{e}^{-2Ls}\,\mathrm{d}s.
\end{equation}
It is negative at $0$ and has nonvanishing differential.
The conformal connection formula gives
\[
 \operatorname{Hess}_\gamma t
 =\mathrm{e}^{-2Lt_0}\left(\operatorname{Hess}_{\widehat g_{1,\lambda}}t_0
       -L|\mathrm{d}t_0|_{\widehat g_{1,\lambda}}^2\widehat g_{1,\lambda}\right).
\]
The norm of $\mathrm{d}t_0$ has a fixed positive lower bound, while the first Hessian has a fixed upper bound on the chosen cylinder.
Choose $L$ sufficiently large and then a preliminary upper bound on $\lambda$ so that
\begin{equation}\label{eq:negative-Hessian}
 \operatorname{Hess}_\gamma t\leq-c\gamma
\end{equation}
with one $c>0$.

For an integer $p\geq2$ and an amplitude $R_*>0$, let
\begin{equation}\label{eq:radii}
 \begin{gathered}
 \ell_\delta(t)=\delta\log(1+\mathrm{e}^{t/\delta}),\qquad
 r_\delta(y)=R_*\ell_\delta(t(y))^p,\\
 F_\delta=\log r_\delta,\qquad
 \beta_\delta=|\mathrm{d}F_\delta|_\gamma,\qquad0<\delta\leq1.
 \end{gathered}
\end{equation}
The following bounds hold uniformly in $0<\delta\leq1$, with the fixed background geometry:
\begin{equation}\label{eq:profile-bounds}
 \beta_\delta\geq cp,\qquad
 \operatorname{Hess}_\gamma F_\delta\leq-c\beta_\delta\gamma,
 \qquad |\nabla^jF_\delta|\leq C_j\beta_\delta^j
       \quad(1\leq j\leq8).
\end{equation}
For clarity, if $a_\delta=(\log\ell_\delta)'$, then
\[
 a_\delta(t)=\delta^{-1}
  \frac{\mathrm{e}^{t/\delta}}
       {(1+\mathrm{e}^{t/\delta})\log(1+\mathrm{e}^{t/\delta})}>0,
 \qquad a_\delta'\leq0.
\]
It follows from~\eqref{eq:negative-Hessian} that
\[
 \operatorname{Hess}_\gamma F_\delta
 =p a_\delta\operatorname{Hess}_\gamma t
        +p a_\delta'\,\mathrm{d}t\otimes \mathrm{d}t
 \leq-c\beta_\delta\gamma.
\]
To verify all derivative allowances, write $a_\delta(t)=\delta^{-1}A(t/\delta)$ with
\[
 A(s)=\frac{\mathrm{e}^s}{(1+\mathrm{e}^s)\log(1+\mathrm{e}^s)}.
\]
For $z=\mathrm{e}^s$, direct differentiation gives
\[
 A'(s)=\frac{z(\log(1+z)-z)}{(1+z)^2\log(1+z)^2}\leq0.
\]
At the negative end $A(s)=1+O(\mathrm{e}^s)$, including differentiated expansions.
At the positive end $A(s)$ is comparable to $(1+s)^{-1}$ and its $j$th derivative is $O((1+s)^{-j-1})$.
On the remaining compact interval $A$ is positive.
Hence $|A^{(j)}(s)|\leq C_jA(s)^{j+1}$.
After rescaling, this implies $|\partial_t^j\log\ell_\delta|\leq C_ja_\delta^j$.
On the fixed depth interval $a_\delta$ has a uniform positive lower bound.
Since $|\mathrm{d}t|_\gamma$ is bounded above and below, $\beta_\delta=pa_\delta|\mathrm{d}t|_\gamma$ gives the lower bound and the derivative bounds in~\eqref{eq:profile-bounds} by the chain rule.
The constants can be taken independent of $p\geq2$, since $pa_\delta^j\leq p^ja_\delta^j$.
Also $a_\delta\leq\ell_\delta^{-1}$ and $|\mathrm{d}t|_\gamma$ is bounded, so
\begin{equation}\label{eq:radius-smallness}
 r_\delta\beta_\delta^2\leq C R_*p^2\ell_\delta^{p-2}.
\end{equation}
The right side is uniformly small after choosing $p$ and then decreasing $R_*$.
The radii are increasing in $\delta$ and converge uniformly to zero on compact subsets of $\{t<0\}$.
We use only this convergence, with the spatial dilation fixed.

Choose $t_b>0$ small enough that $3t_b<t|_{t_0=1/4}$.
Let
\[
 K_0=\{y\in Y:y_n\geq-1/2,\ t(y)\leq2t_b\}.
\]
A smooth cutoff $\chi\in C_c^\infty(Y^\circ)$ can be chosen equal to one near $K_0$ and satisfying
\begin{equation}\label{eq:cutoff-support}
 \operatorname{supp}(\mathrm{d}\chi)
   \subset\{y_n<-1/4\}\cup\{t>t_b\}.
\end{equation}
For example, multiply a lower-height cutoff, transitioning between $-3/4$ and $-1/2$, by a cutoff of $t$ that transitions between $2t_b$ and $3t_b$.
The choice of $t_b$ keeps the support away from the top and lateral faces of $Y$.
Every derivative of the cutoff is therefore either in the known lower region, or in a region where
\begin{equation}\label{eq:positive-cutoff-radius}
 r_\delta\geq R_*t_b^p>0.
\end{equation}
All these choices are independent of $\delta$.

\subsection{Transfer functionals on shrinking spheres}\label{subsec:transfer}
\begin{lemma}\label{lem:transfer}
Under the local data specified in Section~\ref{subsec:sphere-geometry}, the parameters in~\eqref{eq:radii} and one fixed spatial dilation can be chosen so that, for every $0<\delta\leq1$, the mixed kernel extends to a neighborhood of
\[
 \{(x,y):y\in Y,\ d_\gamma(x,y)=r_\delta(y)\}.
\]
There is a smooth density $\psi_\delta(y,\omega)$ on the bundle of $\gamma_y$-unit spheres such that, for every local first-metric harmonic function $v$,
\begin{align}\label{eq:transfer-representation}
 T_{\delta,y}v
 &=\int_{\partial B_\gamma(y,r_\delta(y))}
       (K\partial_{\nu_1}v-v\partial_{\nu_1}K)\,\mathrm{d}S_1\notag\\
 &=\int_{S_y}\psi_\delta(y,\omega)
       v(\exp_y^\gamma(r_\delta(y)\omega))\,\mathrm{d}\sigma_y(\omega).
\end{align}
Here $\nu_1,\mathrm{d}S_1$ refer to the actual first metric $\widehat g_{1,\lambda}$, and $\mathrm{d}\sigma_y$ is the probability measure on $S_y$.
Locally in its center variable, $T_{\delta,y}v$ is second-metric harmonic.
On the entire center cylinder, the given finite harmonic pairs satisfy
\begin{equation}\label{eq:reproduction}
 T_{\delta,y}1=1,\qquad
 T_{\delta,y}v_{1,\lambda}^a=v_{2,\lambda}^a(y),\qquad
 \int_{S_y}\psi_\delta(y,\omega)\,\mathrm{d}\sigma_y(\omega)=1.
\end{equation}
Every fixed angular smooth norm of $\psi_\delta$ and its first base derivatives is uniformly bounded on a fixed neighborhood of the supports of derivatives of $\chi$.
\end{lemma}
\begin{proof}
\opSevenStep{1. What is needed from the original kernels}
The local kernels satisfy the homogeneous equations away from the diagonal, have their metric point-mass normalization, and have uniform rescaled derivative bounds at fixed positive separation by~\eqref{eq:kernel-transform} and~\eqref{eq:spatial-dilation}.
Their cross data agree by~\eqref{eq:initial-cross}.
Lemma~\ref{lem:fixed-separation} constructs the required fixed-separation extensions.
Its proof permits smooth lower-order coefficients, so it applies in the original common coordinates used here.
The argument concerns only the selected local neighborhoods.

\opSevenStep{2. The strict moving-sphere test}
We verify the strict criterion of Lemma~\ref{lem:strict-product} directly for the present conformal pair.
Fix an intermediate radius $r=r_{\delta'}$, with $\delta\leq\delta'\leq1$, and abbreviate $F=\log r$, $\beta=|\mathrm{d}F|_\gamma$ and $b=\nabla^\gamma r$.
The defining function is
\[
 \rho(x,y)=\tfrac12\bigl(d_\gamma(x,y)^2-r(y)^2\bigr),
\]
positive outside the sphere.
Parallel translation along the short radial geodesic identifies the endpoint tangent spaces.
Let $\omega$ be its unit direction.
Because the two normalized principal metrics coincide by~\eqref{eq:H-zero}, the vectors entering~\eqref{eq:strict-product-condition} can be normalized as
\[
 U_1=\omega+\mathrm{i}  z,\quad z\perp\omega,\quad |z|=1,
 \qquad
 U_2=\frac{\omega+b}{|\omega+b|^2}+\mathrm{i}  z',
 \quad z'\perp(\omega+b),\quad |z'|=|\omega+b|^{-1}.
\]
In particular $(\omega+b)\cdot U_2=1$.
For a real quadratic form we use the Hermitian evaluation on these complex vectors, namely the sum of its values on the real and imaginary parts.

The Hessian of half the squared distance differs from its Euclidean endpoint expression by $O(r^2)(|U_1|^2+|U_2|^2)$.
This follows from the Jacobi equation with prescribed endpoints in a parallel frame; the field differs from the affine field by $O(r^2)$, including its first derivative.
Since $|b|=r\beta$ is small, both partial gradients of $\rho$ are nonzero.
The normal component identity
\[
 \omega\cdot(U_1-U_2)=b\cdot U_2
\]
and $\operatorname{Hess}(r^2/2)=2\,\mathrm{d}r\otimes \mathrm{d}r+r^2\operatorname{Hess}F$ show that the strict expression is
\begin{equation}\label{eq:moving-strict-expression}
 |\pi_{\omega^\perp}(U_1-U_2)|^2-|b\cdot U_2|^2
       -r^2\operatorname{Hess}_\gamma F(U_2,\overline{U_2})+O(r^2).
\end{equation}
Here $\pi_{\omega^\perp}$ is orthogonal projection.

The real part of the first projection is $-\pi_{\omega^\perp}b+O(|b|^2)$.
For its imaginary part, the lengths of $z$ and $\pi_{\omega^\perp}z'$ differ by $\omega\cdot b+O(|b|^2)$.
Therefore the first squared term is at least $|b|^2-C|b|^3$.
The second is at most $(1+C|b|)|b|^2$, since the real and imaginary parts of $U_2$ are orthogonal, with squared lengths $1+O(|b|)$.
Finally~\eqref{eq:profile-bounds} and the uniform lower bound on $|U_2|$ give
\[
 \eqref{eq:moving-strict-expression}
       \geq c r^2\beta-Cr^2-Cr^3\beta^3.
\]
Choose $p$ so that the lower bound for $\beta$ absorbs the second term, and then choose $R_*$ so that~\eqref{eq:radius-smallness} absorbs the third.
The strict condition follows at every intermediate radius.
There is no discrepancy term because $H=0$ identically.

\opSevenStep{3. Existence and compatible gluing for each positive radius}
Fix these profile parameters and the cutoff.
At $\delta'=1$ all relevant radii have a positive minimum.
On the second set of~\eqref{eq:cutoff-support} the same is true by~\eqref{eq:positive-cutoff-radius}.
On its first set, the lower cross gives the original first kernel.
Choose one spatial dilation $\lambda$ small enough for fixed-separation initialization on all these fixed ranges and so that every required coordinate set lies inside the original neighborhood $V$.
The smooth coefficients on those rescaled sets then have fixed bounds.
Keep this $\lambda$ fixed.

The gluing argument in Section~\ref{subsec:product-surfaces}, applied to $\delta\leq\delta'\leq1$, constructs compatible germs over the whole moving-sphere family, agreeing with the initialized kernel and the lower cross.
For each positive $\delta$ this is a compact family at positive separation and the strict test just verified applies.
This proves existence and smoothness for every positive $\delta$.
It gives no uniform bound for the kernel as $\delta\downarrow0$; such a bound is not assumed below.
The boundary of the center cylinder is already controlled.
The bottom face is in the lower region, and the top face has $t>t_b$.
On a lateral face $|y'|=1$, either $y_n<-1/4$ or $t_0=y_n+1\geq3/4$, so one of the same controlled regions applies.
Use a slightly enlarged cylinder to obtain open neighborhoods of these faces.
For fixed $\delta>0$, the radius and strict margin have positive minima on $Y\times[\delta,1]$.
Constructions with different lower endpoints agree on their common parameter interval: they start from the same initial neighborhood at $\delta'=1$, and the overlap comparisons identify each successive continuation.
They therefore give one compatible family of kernels.

\opSevenStep{4. Representing the flux by a density}
For a fixed center and radius, let $D_y$ be the Dirichlet-to-Neumann operator of the first metric on the ball.
In the first term of the flux in~\eqref{eq:transfer-representation}, move $D_y$ from the boundary trace of $v$ to the smooth boundary trace of $K$ by its surface adjoint.
The resulting expression is integration of the boundary trace of $v$ against a smooth density.
Pull this density back by $\omega\mapsto\exp_y^\gamma(r(y)\omega)$, including the surface Jacobian relative to the probability measure $\mathrm{d}\sigma_y$.
This defines $\psi_\delta$.
The Dirichlet Laplacian on each ball is invertible, and smooth dependence of this local Dirichlet problem gives smooth center dependence for each positive $\delta$.
Explicitly, define the positive Jacobian by $(\exp_y^\gamma(r\,\cdot))^*\mathrm{d}S_1=J_{y,r}\,\mathrm{d}\sigma_y$.
Green's identity makes $D_y$ self-adjoint for $\mathrm{d}S_1$, so
\[
 \psi_\delta(y,\omega)
 =J_{y,r}(\omega)
 \bigl(D_y(K(\cdot,y)|_{\partial B})-\partial_{\nu_1}K(\cdot,y)\bigr)
               (\exp_y^\gamma(r\omega)).
\]
This expression includes the surface Jacobian and the conormal factor and fixes the sign of the density.

To prove second-variable harmonicity, fix $y_0$ and a first-metric harmonic $v$ defined near its closed ball.
The continued kernel is defined in an open neighborhood of the corresponding compact sphere.
Choose a fixed slightly larger surrounding surface within this neighborhood and the domain of $v$.
For centers sufficiently close to $y_0$, both moving spheres and the intervening annuli lie in the same kernel neighborhood.
Green's identity on each annulus changes the moving flux to the flux on that fixed surface.
The second-variable operator may now be applied under the integral and gives zero.
This argument only requires a sufficiently small center neighborhood, which may depend on $v$ and $\delta$.

On the lower known region the kernel is $K_{1,\lambda}$.
Its local fundamental-solution normalization and Green's identity give $T_{\delta,y}v=v(y)$.
In particular the constant function is a legitimate harmonic function for both transformed Laplacians.
Each $T_{\delta,y}v_{1,\lambda}^a-v_{2,\lambda}^a(y)$ is second-metric harmonic on the connected center cylinder and vanishes on its known lower part.
Unique continuation gives~\eqref{eq:reproduction}.
The identity $T_{\delta,y}1=1$ is exactly the last normalization in that equation.
No positivity of $\psi_\delta$ is claimed or used.

\opSevenStep{5. Bounds where the cutoff changes}
On the known lower part the transfer is ordinary evaluation at the center of the ball.
Rescale this ball to the unit tangent ball and multiply the equation by its Laplacian scaling factor.
After positive scalar normalization, its principal coefficients tend smoothly to the Euclidean ones, and its first-order coefficients are $O(r)$.
The center evaluation density for this uniformly invertible Dirichlet family is smooth in the center and radius down to $r=0$, with uniform angular smooth norms.
Differentiating its center dependence and substituting $r=r_\delta(y)$ costs at most $|\mathrm{d}r|=r\beta$, which is bounded.
On the other cutoff region, $r$ has the fixed positive lower bound~\eqref{eq:positive-cutoff-radius}; all required radius derivatives are uniformly bounded there.
Fixed-separation kernel estimates and the smooth ball Dirichlet problem give the same conclusion.
Compactness of the cutoff derivative supports proves the asserted uniform bounds.

The proof is complete.
\end{proof}

\section{Uniform bounds for the transfer density}\label{sec:density}
\subsection{The density equation and its operator decomposition}\label{subsec:density-equation}
The density used in this section is $\mathrm{d}V_\gamma(y)\mathrm{d}\sigma_y(\omega)$, where each sphere has total measure one.
Horizontal differentiation $\nabla$ is defined by parallel transport for $\gamma$, including the transport of tensor indices.
This connection preserves the sphere measure.
A star denotes the adjoint on the sphere alone; no base-variable integration is included in that adjoint.

For a function on the $\gamma_y$-unit sphere, extend it harmonically to the Euclidean tangent unit ball with metric $\gamma(y)$.
Let $B$ be its boundary radial derivative, and let $A_i$ be the boundary value of the $i$th ambient derivative.
On spherical harmonics of degree $k$, $B=k$ and $A_i$ lowers degree by one.
Consequently
\begin{equation}\label{eq:angular-identities}
 [B,A_i^*]=A_i^*,\qquad \sum_iA_i^*A_i^*=0.
\end{equation}
The second identity is the adjoint of the vanishing Laplacian of a harmonic extension.
The fiber Sobolev norm is
\[
 \|z(y,\cdot)\|_s=\|(1+B)^s z(y,\cdot)\|_{L^2(S_y)}.
\]
These norms are equivalent to the usual sphere Sobolev norms.

Suppress the subscript $\delta$ for the moment.
For a local first-metric harmonic $v$, set
\[
 V(y,\omega)=v(\exp_y^\gamma(r(y)\omega)).
\]
Differentiating this trace and using the Dirichlet problem on the rescaled tangent ball gives sphere operators with
\begin{equation}\label{eq:trace-operators}
 \nabla_i V=C_iV,\qquad
 C_i^*=r^{-1}A_i^*+F_iB+E_i+J_i,
 \qquad F_i=\nabla_i\log r,
 \qquad E_i\in r\Psi^1,\quad J_i\in\Psi^0.
\end{equation}
Here $\Psi^m$ denotes a classical pseudodifferential operator of order $m$ on the sphere.
The normalized remainders have uniform angular symbol bounds.
Their center derivatives through order six have the factor allowances in~\eqref{eq:profile-bounds}.
A fixed finite set of higher angular seminorms is used for the Sobolev operator bounds below.

We explain why no coordinate or potential condition is hidden in this application.
In $\gamma$-normal coordinates centered at $y$, positive scalar multiplication gives the first equation principal matrix $\gamma^{-1}$.
On the rescaled unit ball its principal coefficients are $I+O(r^2)$, and its first-order coefficients are $O(r)$.
Its radial derivative operator has the expansion
\[
 B+rB_1(y)+r^2B_2(y,r),\qquad B_1\in\Psi^0,\quad B_2\in\Psi^1.
\]
The first derivative at $r=0$ has order zero because the first $r$-derivative of the principal coefficients vanishes.
The smooth boundary factorization~\cite{lee1989determining} and Taylor expansion in the classical symbol class give the stated operator expansion.
On the fixed unit ball, uniform invertibility and elliptic regularity~\cite{mclean2000strongly} control the smoothing correction.
Differentiating the inverse identity gives the required center and radius bounds, so the expansion holds for the complete operators.
Differentiating the exponential map with its direction parallel transported gives the velocity
\[
 r^{-1}(e_i+O(r^2))+F_i\omega.
\]
Combining this velocity with the radial derivative expansion gives~\eqref{eq:trace-operators}; for example, $rF_iB_1$ is bounded because $r\beta$ is small.
This proves the expansion for the actual local harmonic equation.
Our equations retain their smooth drifts, but no zeroth-order potential remains after~\eqref{eq:exact-transform}.
We have not assumed that the fixed coordinate functions are harmonic.

Write the positively normalized second equation in the reference connection as
\[
 Q^{ij}\nabla_i\nabla_j+b_2^i\nabla_i,
 \qquad Q=\gamma^{-1}.
\]
The vector coefficient $b_2$ has fixed smooth bounds after the one spatial dilation.
Differentiate the density representation of $T_{\delta,y}v$.
Metric compatibility and~\eqref{eq:trace-operators} give
\[
 \nabla_iT_{\delta,y}v
   =\int_{S_y}(\nabla_i+C_i^*)\psi_\delta\,V\,\mathrm{d}\sigma_y.
\]
Applying a second derivative in the same way, including its covariant tensor index, and using Lemma~\ref{lem:transfer} gives
\begin{equation}\label{eq:density-PDE}
 \mathcal{P}_\delta\psi_\delta=0,\qquad
 \mathcal{P}_\delta
   =r\left[Q^{ij}(\nabla_i+C_i^*)(\nabla_j+C_j^*)
                         +b_2^i(\nabla_i+C_i^*)\right].
\end{equation}
The coefficients $Q^{ij}$ stand outside the differentiated product.
In particular they are not differentiated by an unintended full adjoint.

Initially this equation pairs to zero with traces of functions harmonic near the fixed closed ball.
Those traces are dense in smooth boundary data.
To see this, solve the first-metric Dirichlet problem on balls with slightly larger radii, prescribe smooth boundary data transported from the limiting unit sphere, and let the excess radius tend to zero.
After rescaling to the unit ball, smooth Dirichlet dependence gives convergence of the restrictions and every fixed boundary derivative.
Thus the smooth quantity in~\eqref{eq:density-PDE} pairs to zero with a dense set, which proves the equation itself.
The argument is applied at each fixed center separately.
Every approximating solution is harmonic near the fixed closed ball, so the differentiation identity applies there.
The final pairing uses only its boundary trace, and convergence of that trace suffices.
No approximation of center derivatives of the harmonic test functions is required.

We prove the estimate for the full operator~\eqref{eq:density-PDE}.
The parameters are chosen so that $\beta\geq\beta_0$ and $r\beta^2\leq\varepsilon$, and the normalized second principal matrix is exactly $Q=\gamma^{-1}$.
Define
\[
 D_i^+=\nabla_i+F_iB,\qquad D_i^-=-\nabla_i+F_iB,
 \qquad \mathcal{T}=\sum_iA_i^*D_i^+,
 \qquad \mathcal{T}^\flat=\sum_iA_iD_i^-.
\]
The notation $T_{\delta,y}$ remains reserved for the Green functional.
The operators $\mathcal{T}$ and $\mathcal{T}^\flat$ act on functions on the sphere bundle.

The commutator identity~\eqref{eq:angular-identities} gives
\[
 D_i^+(r^{-1}A_j^*)=r^{-1}A_j^*D_i^+,
 \qquad \sum_iA_i^*A_i^*=0.
\]
In the first identity, differentiating $r^{-1}$ contributes $-F_ir^{-1}A_j^*$, while commuting $B$ across $A_j^*$ contributes the opposite term.
Expanding~\eqref{eq:density-PDE} by~\eqref{eq:trace-operators} consequently gives
\begin{equation}\label{eq:P-decomposition}
 \mathcal{P}=2\mathcal{T}+U+J,
 \qquad U=\sum_{|\alpha|\leq2}U_\alpha(y)\partial_y^\alpha,
 \quad U_\alpha\in\varepsilon\beta^{-|\alpha|}\Psi^{2-|\alpha|},
 \quad J\in\Psi^1.
\end{equation}
Here $J$ has no center derivatives and has uniformly bounded angular symbol seminorms.
The notation $a\Psi^m$ means that division by the positive scalar weight $a(y)$ leaves a bounded family in the finitely many angular symbol seminorms used.
For center derivatives of order $j\leq5$, the coefficient bounds permit an additional factor $C_j\beta^j$.
All relevant smooth metric bounds are fixed before $\delta$ varies.

We justify each part of this decomposition.
The angular contraction with two copies of $r^{-1}A_i^*$ vanishes because $Q=\gamma^{-1}$.
The cross terms with $D_i^+$ give exactly $2\mathcal{T}$ by the displayed commutation.
The term $rQ^{ij}D_i^+D_j^+$ has coefficient sizes $O(r)$, $O(r\beta)$ and $O(r\beta^2)$ at center derivative orders two, one and zero.
These satisfy the bounds for $U$ because $r\beta^2\leq\varepsilon$.
Every term containing $E_i\in r\Psi^1$ satisfies the same bounds; its cross term with $r^{-1}A_j^*$ has size $O(r)\Psi^2$.
The cross terms of $J_i\in\Psi^0$ with $r^{-1}A_j^*$, after the outside multiplication by $r$, are bounded angular operators of order one and are placed in $J$.
Their other terms contain a factor $r$ or $r\beta$ and belong to $U$.
Finally the drift paired with $A_i^*$ is a bounded order-one term in $J$, and the remaining drift terms again have a factor $r$ or $r\beta$.
Differentiating these expressions gives the stated coefficient allowances.
In particular, the entire second-center-derivative term is
\begin{equation}\label{eq:second-base-coefficient}
 rQ^{ij}\partial_{y_i}\partial_{y_j}.
\end{equation}
Its coefficient is a real symmetric matrix of scalars independent of the sphere variable.
This fact will cancel the possible third-center-derivative terms below.

For a sphere function $\varphi$, define
\begin{align}\label{eq:N-norm}
 \mathcal{N}_\delta(\varphi)^2
 =\int_Y\bigl(&\|\nabla\varphi\|_0^2+\beta^2\|\varphi\|_1^2\notag\\
       &+\beta^{-1}\|\nabla\varphi\|_{1/2}^2
         +\beta\|\varphi\|_{3/2}^2\bigr)\,\mathrm{d}V_\gamma.
\end{align}
The subscripts denote the sphere Sobolev norms defined above.
The last two terms retain the angular derivatives needed to control the remainder $U$ in the complete equation.

\subsection{A weighted estimate on spherical harmonics}
\begin{lemma}\label{lem:raising}
Let $n\geq3$, and let the profile satisfy~\eqref{eq:profile-bounds}.
On a fixed center region with bounded reference geometry, there exist $\beta_0$ and $c>0$ such that, provided $\beta_\delta\geq\beta_0$ throughout that region,
\begin{equation}\label{eq:raising-estimate}
 \|\mathcal{T}\varphi\|_2^2-\|\mathcal{T}^\flat\varphi\|_2^2
       \geq c\,\mathcal{N}_\delta(\varphi)^2
\end{equation}
for every smooth function compactly supported in the center variable and with zero mean on each sphere.
The constants are independent of $\delta$.
\end{lemma}
\begin{proof}
\opSevenStep{1. Polynomial norms and the degree factors}
Each operator changes the angular degree by exactly one, so different input degrees have orthogonal images for either operator.
It suffices to work on a fixed degree $k\geq1$.
Identify that component with a harmonic homogeneous polynomial $v(z)$ of degree $k$ on a tangent space.
In the Fischer inner product, monomials $z^\alpha$ are orthogonal and have squared norm $\alpha!$.
Differentiation $a_i=\partial_{z_i}$ is adjoint to multiplication $M_i=z_i$ on the full polynomial space.

On harmonic homogeneous polynomials the Fischer squared norm is the normalized sphere squared norm multiplied by
\[
 N_k=n(n+2)\cdots(n+2k-2),\qquad N_0=1.
\]
One direct verification uses integration with density $(2\pi)^{-n/2}\mathrm{e}^{-|z|^2/2}$.
For harmonic homogeneous $v$, the differentiated exponential satisfies
$v(\partial_z)\mathrm{e}^{-|z|^2/2}=(-1)^kv(z)\mathrm{e}^{-|z|^2/2}$;
the trace terms vanish because $\Delta_zv=0$.
Integration by parts identifies this integral inner product with the Fischer one on degree $k$.
Polar integration gives the radial moment $N_k$, proving the norm formula.

Set $d=k+n/2-1$.
If $q_i$ are harmonic of degree $k$ and $s=\sum_i a_iq_i$, the harmonic projection of $\sum_iM_iq_i$ is
\[
 \Pi_{k+1}\sum_iM_iq_i
       =\sum_iM_iq_i-\frac{|z|^2}{2d}s.
\]
Indeed, $s$ is harmonic of degree $k-1$,
$\Delta_z\sum_iM_iq_i=2s$, and $\Delta_z(|z|^2s)=4ds$.
The second term is orthogonal to harmonic polynomials because multiplication by $|z|^2$ is adjoint to $\Delta_z$.
Its squared norm is $d^{-1}\|s\|_F^2$, where $F$ denotes the Fischer norm.
Expanding the first squared norm using the multiplication and differentiation adjoints gives
\begin{equation}\label{eq:harmonic-projection-norm}
 \left\|\Pi_{k+1}\sum_iM_iq_i\right\|_F^2
 =\sum_i\|q_i\|_F^2+
     \sum_{i,j}(a_jq_i,a_iq_j)_F
       -\frac1d\left\|\sum_i a_iq_i\right\|_F^2.
\end{equation}
Since $N_{k+1}/N_k=2d+2$ and $N_k/N_{k-1}=2d$, the sphere adjoint is
$A_i^*|_{\mathcal H_k}=(2d+2)\Pi_{k+1}M_i$.

At degree $k$, write $D_i^\pm=\pm\nabla_i+kF_i$.
Formula~\eqref{eq:harmonic-projection-norm} shows that $N_k$ times the left side of~\eqref{eq:raising-estimate} is
\begin{equation}\label{eq:raising-expanded}
 (2d+2)\left[\|D^+v\|_F^2+
    \sum_{i,j}(a_jD_i^+v,a_iD_j^+v)_F-\frac1d\left\|\sum_i a_iD_i^+v\right\|_F^2\right]
       -2d\left\|\sum_i a_iD_i^-v\right\|_F^2.
\end{equation}
All pairings here include integration in the center variable.
We take real parts of the cross pairings when the functions are complex.

\opSevenStep{2. Integration by parts and curvature terms}
Let $C_{ij}=2k(\operatorname{Hess}_\gamma F_\delta)_{ij}$ and define
\[
 C[av]=\int_Y\sum_{i,j}C_{ij}(a_iv,a_jv)_F\,\mathrm{d}V_\gamma.
\]
Covariant integration by parts yields
\begin{align}
 \sum_{i,j}(a_jD_i^+v,a_iD_j^+v)_F
  &=\left\|\sum_i a_iD_i^-v\right\|_F^2-C[av]
          +O(k^2)\|v\|_F^2,\label{eq:raising-ibp-one}\\
 \left\|\sum_i a_iD_i^+v\right\|_F^2
  &\leq k\|D^-v\|_F^2-C[av]+O(k^2)\|v\|_F^2,
       \label{eq:raising-ibp-two}\\
 \|D^-v\|_F^2
  &=\|D^+v\|_F^2+\int_Y\operatorname{tr}C\,|v|_F^2\,\mathrm{d}V_\gamma.
       \label{eq:raising-ibp-three}
\end{align}
To check the first identity, move a $D_i^+$ across the pairing and commute $D_i^-D_j^+$ into $D_j^+D_i^- -C_{ij}$, including the curvature commutator of the horizontal covariant derivatives.
The reordered pairing is the displayed divergence square.
For the second identity the same operations leave a pairing of $(a_iD_j^-v)_{i,j}$ with its transposition.
Cauchy--Schwarz bounds it by $k\|D^-v\|_F^2$, since
$\sum_i\|a_iw\|_F^2=k\|w\|_F^2$
on degree $k$.
A curvature commutator acts as a sum of rotations on degree-$k$ polynomials, of norm $O(k)$; the two contractions contribute one further factor $k$.
Thus both error bounds are $O(k^2)\|v\|_F^2$ with fixed geometric constants.
Expanding the two squares and integrating the cross terms proves~\eqref{eq:raising-ibp-three}.

Substitution of~\eqref{eq:raising-ibp-one} into~\eqref{eq:raising-expanded} gives
\begin{align}\label{eq:raising-intermediate}
 (2+2/d)\left[d\|D^+v\|_F^2-
       \left\|\sum_i a_iD_i^+v\right\|_F^2-dC[av]\right]
 +2\left\|\sum_i a_iD_i^-v\right\|_F^2
       +O(k^3)\|v\|_F^2.
\end{align}

\opSevenStep{3. Retaining both derivative strengths}
Choose a fixed small $\vartheta>0$.
Estimate the fraction $1-\vartheta$ of the negative divergence square in~\eqref{eq:raising-intermediate} by~\eqref{eq:raising-ibp-two}--\eqref{eq:raising-ibp-three}.
For the remaining fraction use the pointwise convex combination
\begin{align}\label{eq:raising-convex-bound}
 \left|\sum_i a_iD_i^+v\right|_F^2
 \leq{}&(1-\zeta)(k+n-1)|D^+v|_F^2\notag\\
       &+\zeta\left(2\left|\sum_i a_iD_i^-v\right|_F^2
                         +8k^3\beta^2|v|_F^2\right),
 \qquad\zeta=c_1/\beta.
\end{align}
For the first pointwise bound, consider the contraction on the full homogeneous polynomial spaces, where its adjoint is multiplication by $(z_i)$ and
$\sum_i a_iM_i=k+n-1$
on degree $k-1$.
The squared operator norm is therefore at most $k+n-1$.
Restriction to harmonic polynomials can only decrease the norm.
The precise adjoint on the harmonic subspaces includes harmonic projection; multiplication itself need not preserve harmonicity.
For the second bound use $D_i^++D_i^-=2kF_i$ and
$|\sum_iF_ia_iv|_F^2\leq k\beta^2|v|_F^2$.
Choosing $\beta_0$ large makes $0\leq\zeta\leq1$.
The variable coefficient $\zeta$ is used only in the pointwise inequality, so no derivative of it is introduced.

After these substitutions the coefficient of $|D^+v|_F^2$ is
\begin{equation}\label{eq:raising-retained-coefficient}
 (2+2/d)\,[d-k-\vartheta(n-1)+\vartheta\zeta(k+n-1)].
\end{equation}
Because $d-k=(n-2)/2>0$, choose $\vartheta$ small so that this bounds below a positive constant plus a positive multiple of $c_1k/\beta$.
The coefficient of $|\sum_i a_iD_i^-v|_F^2$ is nonnegative after increasing $\beta_0$.
The Hessian contribution is exactly
\begin{equation}\label{eq:raising-Hessian-form}
 -(2+2/d)\left[(d-1+\vartheta)C[av]
       +(1-\vartheta)k\int_Y\operatorname{tr}C\,|v|_F^2\,\mathrm{d}V_\gamma\right].
\end{equation}
The two coefficients are positive for $n\geq3$, $k\geq1$.
The inequality
$\operatorname{Hess}_\gamma F_\delta\leq-c_0\beta\gamma$
and the contraction identity imply that~\eqref{eq:raising-Hessian-form} is at least
$c k^3\int_Y\beta|v|_F^2\,\mathrm{d}V_\gamma$.
The choices can be made uniformly in $k$.
Take $0<\vartheta<\min\{1/2,(n-2)/(4(n-1))\}$, giving $d-k-\vartheta(n-1)\geq(n-2)/4$.
Since $d-1+\vartheta\geq k/2$, the Hessian form is bounded below by $c_H k^3\int\beta|v|_F^2\,\mathrm{d}V_\gamma$ for a constant $c_H>0$ independent of $k$.
The negative scalar contribution from~\eqref{eq:raising-convex-bound} is at most $C\vartheta c_1 k^3\int\beta|v|_F^2\,\mathrm{d}V_\gamma$.
First fix $c_1>0$ with $C\vartheta c_1<c_H/4$.
Then increase $\beta_0$ to absorb the curvature error and to ensure $\zeta\leq1$ and $2-2(2+2/d)\vartheta\zeta\geq0$ for every $k\geq1$.
The last expression is the remaining coefficient of the divergence square, and its uniform nonnegativity follows from $d\geq3/2$.
We retain positive multiples of
\[
 \|D^+v\|_F^2,\qquad
 \int_Y\frac{k}{\beta}|D^+v|_F^2\,\mathrm{d}V_\gamma,\qquad
 k^3\int_Y\beta|v|_F^2\,\mathrm{d}V_\gamma.
\]

Finally,
\[
 \|D^+v\|_F^2
 =\|\nabla v\|_F^2+k^2\int\beta^2|v|_F^2
               -k\int\Delta_\gamma F_\delta\,|v|_F^2
 \geq\|\nabla v\|_F^2+k^2\int\beta^2|v|_F^2.
\]
For the weighted derivative use the pointwise bound
\[
 \frac{k}{\beta}|\nabla v|_F^2
       \leq2\frac{k}{\beta}|D^+v|_F^2+2k^3\beta|v|_F^2.
\]
Division by $N_k$ gives exactly the four quantities in~\eqref{eq:N-norm} at this degree, up to fixed constants, since $(1+k)^s$ and $k^s$ are comparable for $k\geq1$.
Sum the orthogonal components.
Approximation by finite spherical harmonic sums proves the assertion for every smooth $\varphi$ in the statement.

The proof is complete.
\end{proof}

\subsection{Coercivity and the density bound}
\begin{lemma}\label{lem:coercivity}
For the operator $\mathcal{P}_\delta$ in~\eqref{eq:density-PDE}, with $Q=\gamma^{-1}$, choose $\beta_0$ sufficiently large and $\varepsilon_0>0$ sufficiently small in terms of the fixed smooth coefficients and the constants in~\eqref{eq:profile-bounds}.
If
\begin{equation}\label{eq:coercivity-conditions}
 \beta_\delta\geq\beta_0,\qquad
 r_\delta\beta_\delta^2\leq\varepsilon\leq\varepsilon_0,
\end{equation}
then
\begin{equation}\label{eq:full-coercivity}
 \mathcal{N}_\delta(\varphi)
       \leq C\|\mathcal{P}_\delta\varphi\|_{L^2(S_\gamma Y)}
\end{equation}
for every smooth compactly center-supported $\varphi$ with zero mean on each sphere.
The constant is independent of $\delta$.
\end{lemma}
\begin{proof}
A dagger denotes the full adjoint for $\mathrm{d}V_\gamma(y)\mathrm{d}\sigma_y(\omega)$.
A star still denotes the sphere adjoint only.
Covariant integration and the degree shift $BA_i=A_i(B-1)$ give
\[
 Z:=\mathcal{T}^\flat-\mathcal{T}^\dagger=\sum_iF_iA_i.
\]
The expansion of the two squared norms is the exact identity
\begin{align}\label{eq:squared-norm-identity}
 \|(2\mathcal{T}+U)\varphi\|_2^2
  -\|(2\mathcal{T}^\flat+U^\dagger)\varphi\|_2^2
  =4\bigl(\|\mathcal{T}\varphi\|_2^2
          -\|\mathcal{T}^\flat\varphi\|_2^2\bigr)
             +(\mathcal E\varphi,\varphi),
\end{align}
where
\begin{equation}\label{eq:error-operator}
 \mathcal E=2[\mathcal{T}^\dagger,U]+2[U^\dagger,\mathcal{T}]
                 +[U^\dagger,U]-2\bigl(UZ+(UZ)^\dagger\bigr).
\end{equation}
For example, substitute $\mathcal{T}^\flat=\mathcal{T}^\dagger+Z$ in the mixed terms on the left.
They become the first two commutators and the last two products in~\eqref{eq:error-operator}.
This also verifies the adjoint signs.

We prove the form bound
\begin{equation}\label{eq:error-form-bound}
 |(\mathcal E\varphi,\varphi)|\leq C\varepsilon
                              \mathcal{N}_\delta(\varphi)^2.
\end{equation}
To track all derivative orders, say a coefficient family has indices $(m,s)$ if the coefficient of $\partial_y^\alpha$ has angular order $m-|\alpha|$ and size $O(\beta^{s-|\alpha|})$, with the derivative allowances already specified.
Products add these indices.
Taking a center adjoint differentiates a coefficient, which multiplies its allowance by $\beta$ while removing one center derivative; hence it preserves the indices.
The leading angular symbols here are scalar.
Their commutator loses one angular order, and a center derivative hitting a coefficient likewise loses one total derivative order.
Thus a commutator loses one in the first index and preserves the sum of the second indices.

The indices of $\mathcal{T},\mathcal{T}^\dagger,U,Z$ are respectively
\[
 (2,1),\qquad(2,1),\qquad(2,0),\qquad(1,1).
\]
Every $U$ carries a factor $\varepsilon$, and $\varepsilon\leq1$.
Hence~\eqref{eq:error-operator} has total indices at most $(3,1)$ and a factor $O(\varepsilon)$.
We must also exclude a third derivative in the center variable; a total-order statement alone does not do this.

Use a smooth orthonormal sphere trivialization, and write the center density as $\mu(y)\mathrm{d}y$.
The sphere measure is fixed in this trivialization.
By~\eqref{eq:second-base-coefficient},
\[
 U=a^{ij}\partial_i\partial_j+B^i\partial_i+C,
 \qquad a^{ij}=rQ^{ij}=a^{ji}\in\mathbb{R}.
\]
The coefficients $B^i,C$ are angular operators.
Direct integration by parts gives
\begin{align*}
 U^\dagger-U={}&
  \bigl(2\mu^{-1}\partial_j(\mu a^{ij})-B^i-(B^i)^*\bigr)\partial_i\\
 &+\mu^{-1}\partial_i\partial_j(\mu a^{ij})
    -\mu^{-1}\partial_i\bigl(\mu(B^i)^*\bigr)+C^*-C.
\end{align*}
In particular, this difference has at most one center derivative.
In $[\mathcal{T}^\dagger,U]$ and $[U^\dagger,\mathcal{T}]$, a possible third-center coefficient is a commutator of an angular coefficient with the scalar $a^{ij}$, and therefore is zero.
In $[U^\dagger,U]=[U^\dagger-U,U]$, the same argument cancels the possible third-center coefficient.
The product $UZ$ has at most two center derivatives.
For example, write $U^\dagger-U=V^i\partial_i+V^0$.
The coefficient of a third center derivative in $[U^\dagger-U,U]$ is a sum of $[V^i,a^{jk}]$.
It vanishes because $a^{jk}$ is a scalar function of the center alone.
Fourth center derivatives already canceled in passing to this commutator.
The same calculation applies to the one-center leading terms in the other two commutators; a derivative hitting a coefficient leaves at most two center derivatives.
Converting horizontal derivatives to coordinate derivatives changes only lower center orders and adds bounded angular differential operators.
Therefore every term of $\mathcal E$ belongs to one of the following three classes, or a lower class:
\begin{equation}\label{eq:error-three-types}
 \varepsilon\beta^{-1}\Psi^1\partial_y^2,
 \qquad \varepsilon\Psi^2\partial_y,
 \qquad \varepsilon\beta\Psi^3.
\end{equation}

For the first class integrate once in the center variable.
The main form is bounded by
\[
 C\varepsilon\int_Y\beta^{-1}\|\partial_y\varphi\|_{1/2}^2\,\mathrm{d}V_\gamma.
\]
Differentiating its coefficient costs one factor $\beta$ and leaves an $\varepsilon\Psi^1\partial_y$ term.
This includes the weight derivative: $|\mathrm{d}\beta|\leq C\beta^2$ implies $|\mathrm{d}(\beta^{-1})|\leq C$.
For the second class the angular Sobolev pairing gives
\begin{align*}
 |(\varepsilon\Psi^2\partial_y\varphi,\varphi)|
 &\leq C\varepsilon\int_Y
          \|\partial_y\varphi\|_{1/2}\|\varphi\|_{3/2}\,\mathrm{d}V_\gamma\\
 &\leq C\varepsilon\int_Y
    \bigl(\beta^{-1}\|\partial_y\varphi\|_{1/2}^2
                       +\beta\|\varphi\|_{3/2}^2\bigr)\,\mathrm{d}V_\gamma.
\end{align*}
The third class has form bound
$C\varepsilon\int_Y\beta\|\varphi\|_{3/2}^2\mathrm{d}V_\gamma$.
The lower term produced by the first integration satisfies the same bounds.
The difference between a coordinate derivative and a horizontal derivative is a bounded angular first derivative; its contribution is again controlled by the last two terms of~\eqref{eq:N-norm}, since $\beta\geq1$.
These estimates prove~\eqref{eq:error-form-bound}.
The angular pairings use $\Psi^1:H^{1/2}\to H^{-1/2}$, $\Psi^2:H^{1/2}\to H^{-3/2}$, and $\Psi^3:H^{3/2}\to H^{-3/2}$.
Together with the single center integration by parts, they control all three classes without requiring two center derivatives of $\varphi$ in $\mathcal{N}_\delta$.
They use finitely many angular symbol seminorms and the stated finite center derivative allowances.

Choose $\varepsilon_0$ so that~\eqref{eq:error-form-bound} is smaller than half the positive term supplied by Lemma~\ref{lem:raising} in~\eqref{eq:squared-norm-identity}.
The second squared norm on the left is nonnegative, so
\[
 \mathcal{N}_\delta(\varphi)\leq C\|(2\mathcal{T}+U)\varphi\|_2.
\]
Finally, the bounded order-one family $J$ satisfies
\[
 \|J\varphi\|_2\leq C\|\varphi\|_{L^2(Y;H^1(S_y))}
                   \leq C\beta_0^{-1}\mathcal{N}_\delta(\varphi).
\]
Increase $\beta_0$ to absorb it and use~\eqref{eq:P-decomposition}.
If the preceding estimate has constant $C_*$ and the bound for $J$ has constant $C_J$, take $\beta_0\geq2C_*C_J$.
Then
\[
 \mathcal{N}_\delta(\varphi)
 \leq C_*\|\mathcal{P}_\delta\varphi\|_2
       +\tfrac12\mathcal{N}_\delta(\varphi).
\]
Moving the last term to the left gives a constant $2C_*$ independent of $\delta$.
This proves~\eqref{eq:full-coercivity} for the actual operator, including all its second-center and angular terms.

The proof is complete.
\end{proof}

\begin{lemma}\label{lem:density-bound}
For the actual densities in Lemma~\ref{lem:transfer}, with the parameters fixed as above,
\begin{equation}\label{eq:density-bound}
 \|\chi\psi_\delta\|_{L^2(S_\gamma Y)}\leq C,
 \qquad0<\delta\leq1.
\end{equation}
The constant is independent of $\delta$.
This conclusion follows from Lemma~\ref{lem:coercivity}, using~\eqref{eq:full-coercivity}.
\end{lemma}
\begin{proof}
\opSevenStep{1. The constant component is fixed exactly}
Equation~\eqref{eq:reproduction} gives
\[
 \int_{S_y}(\psi_\delta-1)\,\mathrm{d}\sigma_y=0.
\]
Therefore $\varphi_\delta=\chi(\psi_\delta-1)$ has zero fiber mean and is compactly base-supported.
The number $1$ here is the true constant harmonic function after normalization.
No such assertion has been made for the original Schr\"odinger equations.

\opSevenStep{2. The operator applied to the constant is bounded}
In~\eqref{eq:P-decomposition}, $\mathcal{T}1=0$ because $B1=0$ and $\nabla1=0$.
The terms of $U$ with a base derivative also vanish on $1$.
Its remaining angular operator has uniformly bounded symbol seminorms of order two, with coefficient $O(\varepsilon)$.
Applied to the fixed smooth function $1$, it is bounded in $L^2$ uniformly in $\delta$.
The same holds for $J1$, by its fixed order-one bounds.
Hence
\begin{equation}\label{eq:P-one}
 \|\chi\mathcal{P}_\delta1\|_{L^2}\leq C.
\end{equation}
The cancellation preceding~\eqref{eq:P-decomposition} is important: it prevents a spurious term growing like $\beta_\delta$ in this estimate.

\opSevenStep{3. The commutator uses only controlled cutoff data}
Since $\chi$ depends on the base alone, angular operators commute with it.
Directly,
\[
 [2\mathcal{T},\chi]z
       =2\sum_iA_i^*((\nabla_i\chi)z).
\]
For the second-base term,
\[
 [rQ^{ij}\partial_i\partial_j,\chi]z
       =2rQ^{ij}(\partial_i\chi)\partial_jz
                       +rQ^{ij}(\partial_i\partial_j\chi)z.
\]
The first-base terms of $U$ contribute bounded angular order-one operators on the same fixed cutoff region.
The zero-base term and $J$ commute with $\chi$.
Consequently, on a fixed neighborhood $D_\chi$ of the supports of all derivatives of $\chi$,
\begin{equation}\label{eq:commutator-bound}
 \|[\mathcal{P}_\delta,\chi]z\|_{L^2}
 \leq C_\chi\left(
     \|z\|_{L^2(D_\chi;H^1(S_y))}
          +\|r_\delta\nabla z\|_{L^2(S_\gamma D_\chi)}\right).
\end{equation}
All coefficients in this inequality are bounded independently of $\delta$; for example the mixed coefficients have size $O(r\beta)$, which is bounded under~\eqref{eq:coercivity-conditions}.
Lemma~\ref{lem:transfer} bounds the two norms on the right for $z=\psi_\delta-1$.
This uses only the known lower region and fixed-positive-separation region described in~\eqref{eq:cutoff-support}; no bound for the continued kernel at uncontrolled shrinking separation is substituted here.

\opSevenStep{4. Application of the full operator estimate}
Using~\eqref{eq:density-PDE},
\begin{equation}\label{eq:cutoff-density-equation}
 \mathcal{P}_\delta\varphi_\delta
  =-\chi\mathcal{P}_\delta1
       +[\mathcal{P}_\delta,\chi](\psi_\delta-1).
\end{equation}
Equations~\eqref{eq:P-one} and~\eqref{eq:commutator-bound} bound this right side in $L^2$ uniformly.
Lemma~\ref{lem:coercivity}, through~\eqref{eq:full-coercivity}, gives $\mathcal{N}_\delta(\varphi_\delta)\leq C$.
Since $\beta_\delta\geq1$, the second term of~\eqref{eq:N-norm} controls the ordinary $L^2$ norm.
Thus $\|\varphi_\delta\|_{L^2}\leq C$.
Adding the fixed function $\chi$ proves~\eqref{eq:density-bound}.

The proof is complete.
\end{proof}

\section{Proof of the main theorem}\label{sec:proof}
We first recover the potentials near a contact point and then continue the equality through the connected interior.

\subsection{Local recovery}\label{subsec:local-recovery}
\begin{lemma}\label{lem:local-recovery}
Let $q_1,q_2$ be the two smooth potentials on the fixed extension $N$.
Let $p\notin\overline{U_0}$ be an interior point.
Suppose an open set $W$ contains one side of a smooth hypersurface through $p$ in a coordinate neighborhood and that
\[
 q_1=q_2\hbox{ on }W,\qquad
 G_1=G_2\hbox{ on }W\times W\setminus\operatorname{diag},\qquad
 R_1f=R_2f\hbox{ on }W\quad(f\in C_c^\infty(U_0)).
\]
Then $q_1=q_2$ on a neighborhood of $p$.
\end{lemma}
\begin{proof}
\opSevenStep{1. Normalization on one common coordinate neighborhood}
Choose the $n+1$ sources constructed in Section~\ref{subsec:source-jets}.
Shrink the coordinate neighborhood so that $w_j>1/2$ and the differentials of $v_1^k=u_1^k/w_1$ form a basis.
The transformed objects in~\eqref{eq:normalized-objects} satisfy the local data specified in Section~\ref{subsec:sphere-geometry}.
In particular the normalized metrics and kernels agree on the known side, by the assumed equalities for $G_j$ and $w_j$.
The two transformed metrics are pointwise conformal on the whole common neighborhood, not merely on the known side.

Choose the geometry, profile parameters and one spatial dilation as in Lemma~\ref{lem:transfer}.
The finite family of normalized functions is fixed before these choices.
The density bound of Lemma~\ref{lem:density-bound} is available for every $\delta>0$, with that same spatial dilation.
We now keep all notation in the rescaled coordinate chart until the final step.

\opSevenStep{2. Reproduction of constants gives a first-order remainder}
For each $k$, equations~\eqref{eq:transfer-representation} and~\eqref{eq:reproduction} imply
\begin{align}\label{eq:paired-difference}
 v_{2,\lambda}^k(y)-v_{1,\lambda}^k(y)
   =\int_{S_y}\psi_\delta(y,\omega)
       \bigl[&v_{1,\lambda}^k(\exp_y^\gamma(r_\delta(y)\omega))\notag\\
             &-v_{1,\lambda}^k(y)\bigr]\,\mathrm{d}\sigma_y(\omega).
\end{align}
On the fixed compact set containing these balls, the first function has bounded gradient.
Along the radial geodesic, the fundamental theorem of calculus therefore gives
\[
 |v_{1,\lambda}^k(\exp_y^\gamma(r_\delta\omega))-v_{1,\lambda}^k(y)|
       \leq C_k r_\delta(y).
\]
Cauchy--Schwarz with the probability measure $\mathrm{d}\sigma_y$ gives
\[
 |v_{2,\lambda}^k(y)-v_{1,\lambda}^k(y)|
       \leq C_k r_\delta(y)\|\psi_\delta(y,\cdot)\|_{L^2(S_y)}.
\]
Multiply by $\chi/r_\delta$, square and integrate.
The fixed coordinate volume and $\mathrm{d}V_\gamma$ are comparable, so Lemma~\ref{lem:density-bound} yields
\begin{equation}\label{eq:first-order-moment}
 \left\|\frac{\chi(v_{2,\lambda}^k-v_{1,\lambda}^k)}{r_\delta}
                \right\|_{L^2(Y)}\leq C_k.
\end{equation}
Only constant reproduction was used.
No coordinate reproduction, Hessian spanning condition or second-order Taylor cancellation is needed.

\opSevenStep{3. The radius tends to zero on a fixed neighborhood}
Choose a connected neighborhood $V_0$ of the rescaled contact point with compact closure in $\{\chi=1,t<0\}$.
Such a set exists because $\chi=1$ near $0$ and $t(0)<0$.
On this fixed closure, $t\leq-a$ for some $a>0$.
The inequality $\log(1+s)\leq s$ gives
\[
 \sup_{V_0}r_\delta\leq R_*\delta^p \mathrm{e}^{-pa/\delta}\longrightarrow0.
\]
Equation~\eqref{eq:first-order-moment} consequently gives
\[
 \|v_{2,\lambda}^k-v_{1,\lambda}^k\|_{L^2(V_0)}
      \leq C_k\sup_{V_0}r_\delta\longrightarrow0.
\]
The left side is independent of $\delta$.
It must therefore be zero.
The sources, spatial dilation, common neighborhood and the two functions stay fixed in this limit.
Only the representing spheres and densities vary.
The integrated bound suffices; no pointwise bound or convergence for the densities is needed.
Smoothness gives equality at every point of $V_0$.
After undoing the fixed dilation and shrinking the common connected neighborhood if necessary,
\begin{equation}\label{eq:common-ratios}
 v_1^k=v_2^k=:v^k,\qquad 1\leq k\leq n.
\end{equation}
The original first derivatives remain a basis on this neighborhood by their initial choice.
The neighborhood can be chosen to meet $W$.

\opSevenStep{4. First derivatives determine the normalizing factor}
Equation~\eqref{eq:product-rule} gives, for both $j$,
\[
 \Delta_gv^k+2\langle \mathrm{d}\log w_j,\mathrm{d}v^k\rangle_g=0.
\]
Subtracting,
\begin{equation}\label{eq:log-ratio}
 \langle \mathrm{d}\log(w_1/w_2),\mathrm{d}v^k\rangle_g=0,
                  \qquad k=1,\ldots,n.
\end{equation}
The covectors $\mathrm{d}v^k$ span the cotangent space, and the metric pairing is nondegenerate.
Hence $\mathrm{d}\log(w_1/w_2)=0$.
Its value is constant on the connected neighborhood.
On its intersection with $W$ the paired sources agree, so this constant is zero.
Therefore $w_1=w_2=:w>0$.
Their original homogeneous equations now give
\[
 q_1w=\Delta_gw=q_2w.
\]
Division by the positive function proves $q_1=q_2$ on the neighborhood, as asserted.

The proof is complete.
\end{proof}

The order of parameters in this local proof is worth recording.
First fix the original sources and a neighborhood on which they can be normalized.
Next fix the conformal reference geometry and uniform bounds for a preliminary range of spatial dilations.
Choose $p$ large enough for the geometric and coercive estimates, choose their small parameter $\varepsilon$, and then choose $R_*$ using~\eqref{eq:radius-smallness}.
These choices fix every positive separation needed for initialization and the cutoff bounds.
Choose one smaller spatial dilation $\lambda$ meeting those finitely many requirements and keep it fixed.
Only then let $\delta\downarrow0$.
Since~\eqref{eq:H-zero} holds identically, there is no additional radius-dependent condition on the potential difference to be improved during this limit.

\subsection{Global continuation}\label{subsec:global-continuation}
\begin{proof}[Proof of Theorem~\ref{thm:main}]
\opSevenStep{1. A nonempty region of equality}
By Section~\ref{subsec:boundary-jets} and Lemma~\ref{lem:cap}, form the common fixed nonresonant extension $N$, the exterior collar $E$, and its chosen component $E_0$.
Fix $U_0\Subset E_0$ as above.
Let $W$ be the connected component containing $E_0$ of
\begin{equation}\label{eq:maximal-equality-region}
 \operatorname{int}_{N^\circ}\{x:q_1(x)=q_2(x)\}.
\end{equation}
This is a nonempty connected open set.
The purpose of the remainder of the proof is to show $W=N^\circ$.

\opSevenStep{2. All source solutions agree on this entire region}
For every real $f\in C_c^\infty(U_0)$, the difference $R_1f-R_2f$ vanishes on $E_0$ by Lemma~\ref{lem:cap}.
On $W$ it solves the common homogeneous equation: the source terms cancel, including on $U_0$, and the potentials are equal there.
Unique continuation on the connected set $W$ gives
\begin{equation}\label{eq:source-equality-W}
 R_1f=R_2f\quad\hbox{on }W
                 \qquad(f\in C_c^\infty(U_0)).
\end{equation}

We also need the Green kernels on $W\times W$, not just the source solutions.
Fix $x\in W$.
By~\eqref{eq:source-equality-W} and the common volume density, the kernels agree as distributions in the second variable on $U_0$, and hence pointwise there away from a possible pole.
On $W\setminus\{x\}$ their difference solves the same homogeneous equation in that variable, because the potentials coincide on $W$.
Since $n\geq3$, a path through $x$ can be diverted inside a punctured coordinate ball, so $W\setminus\{x\}$ is connected.
The set $U_0\setminus\{x\}$ contains a nonempty open subset.
Thus
\begin{equation}\label{eq:Green-equality-W}
 G_1(x,y)=G_2(x,y)\quad(x,y\in W,\ x\ne y).
\end{equation}
This step does not require a nonresonant Dirichlet problem on $W$ itself.
It compares the two inverses on the fixed compact extension $N$.

\opSevenStep{3. A smooth touching ball at an interior frontier}
Suppose that $W\ne N^\circ$.
Connectedness of $N^\circ$ implies that the relative frontier of $W$ in $N^\circ$ is nonempty.
Choose $p_0$ on it.
Since $\overline{U_0}\Subset W$, take a coordinate chart centered at $p_0$, disjoint from $\overline{U_0}$, whose image contains $\overline{B(0,R)}$.

Choose $c$ in the coordinate image of $W$ with $|c|<R/4$.
Let
\[
 F=\overline{B(0,R/2)}\setminus W
\]
within that chart.
It is compact and contains $0$.
Since $c\in W$ and $W$ is open, its distance $r=\operatorname{dist}(c,F)$ is positive; also $r\leq|c|<R/4$.
The distance is attained at a point $p\in F$.
The closed ball $\overline{B(c,r)}$ lies in $B(0,R/2)$ because $|c|+r<R/2$.
Its open interior is contained in $W$, while its sphere meets the complement at $p$.

Near $p$, the sphere is a smooth hypersurface and its inner side lies in $W$.
This supplies precisely the one-sided smooth geometry required in Lemma~\ref{lem:local-recovery}.
No smoothness of the whole frontier of~\eqref{eq:maximal-equality-region} has been assumed.

\opSevenStep{4. Elimination of the frontier}
The potentials agree on $B(c,r)$ by construction.
Equations~\eqref{eq:source-equality-W} and~\eqref{eq:Green-equality-W} provide the paired source solutions and the common Green data there.
The touching point is outside $\overline{U_0}$, so Section~\ref{subsec:source-jets} supplies the local normalizing functions and their first derivatives.
Lemma~\ref{lem:local-recovery} gives equality of the potentials on a connected neighborhood of $p$ meeting $B(c,r)$.

This neighborhood lies in the interior of the potential-equality set and meets its component $W$.
It must belong to the same component.
In particular $p\in W$, contradicting its choice in $F$.
Therefore the supposed interior frontier cannot exist, and $W=N^\circ$.

The potentials are consequently equal throughout the interior of the fixed extension.
Smoothness, and hence continuity at the boundary, gives equality on all of $N$.
Restricting to the original $\Omega$ proves~\eqref{eq:main-conclusion}.

The proof is complete.
\end{proof}

\bibliographystyle{plain}
\bibliography{references}

@incollection{calderon1980inverse,
  author = {Calder\'on, Alberto P.},
  title = {On an inverse boundary value problem},
  booktitle = {Seminar on Numerical Analysis and its Applications to Continuum Physics},
  publisher = {Sociedade Brasileira de Matem\'atica},
  address = {Rio de Janeiro},
  pages = {65--73},
  year = {1980}
}

@article{dos2009limiting,
  author = {Dos Santos Ferreira, David and Kenig, Carlos E. and Salo, Mikko and Uhlmann, Gunther},
  title = {Limiting {Carleman} weights and anisotropic inverse problems},
  journal = {Inventiones Mathematicae},
  volume = {178},
  number = {1},
  pages = {119--171},
  year = {2009},
  doi = {10.1007/s00222-009-0196-4}
}

@article{dos2016calderon,
  author = {Dos Santos Ferreira, David and Kurylev, Yaroslav and Lassas, Matti and Salo, Mikko},
  title = {The {Calder\'on} problem in transversally anisotropic geometries},
  journal = {Journal of the European Mathematical Society},
  volume = {18},
  number = {11},
  pages = {2579--2626},
  year = {2016},
  doi = {10.4171/JEMS/649}
}

@article{feizmohammadi2021reconstruction,
  author = {Feizmohammadi, Ali and Krupchyk, Katya and Oksanen, Lauri and Uhlmann, Gunther},
  title = {Reconstruction in the {Calder\'on} problem on conformally transversally anisotropic manifolds},
  journal = {Journal of Functional Analysis},
  volume = {281},
  number = {9},
  pages = {109191},
  year = {2021},
  doi = {10.1016/j.jfa.2021.109191}
}

@book{feldman2025calderon,
  author = {Feldman, Joel and Salo, Mikko and Uhlmann, Gunther},
  title = {The {Calder\'on} Problem: An Introduction},
  series = {Graduate Studies in Mathematics},
  volume = {253},
  publisher = {American Mathematical Society},
  address = {Providence, RI},
  year = {2025},
  doi = {10.1090/gsm/253}
}

@article{koch2005dispersive,
  author = {Koch, Herbert and Tataru, Daniel},
  title = {Dispersive estimates for principally normal pseudodifferential operators},
  journal = {Communications on Pure and Applied Mathematics},
  volume = {58},
  number = {2},
  pages = {217--284},
  year = {2005},
  doi = {10.1002/cpa.20067}
}

@article{lassas2020poisson,
  author = {Lassas, Matti and Liimatainen, Tony and Salo, Mikko},
  title = {The {Poisson} embedding approach to the {Calder\'on} problem},
  journal = {Mathematische Annalen},
  volume = {377},
  pages = {19--67},
  year = {2020},
  doi = {10.1007/s00208-019-01818-3}
}

@article{lee1989determining,
  author = {Lee, John M. and Uhlmann, Gunther},
  title = {Determining anisotropic real-analytic conductivities by boundary measurements},
  journal = {Communications on Pure and Applied Mathematics},
  volume = {42},
  number = {8},
  pages = {1097--1112},
  year = {1989},
  doi = {10.1002/cpa.3160420804}
}

@book{mclean2000strongly,
  author = {McLean, William},
  title = {Strongly Elliptic Systems and Boundary Integral Equations},
  publisher = {Cambridge University Press},
  address = {Cambridge},
  year = {2000}
}

@article{sylvester1987global,
  author = {Sylvester, John and Uhlmann, Gunther},
  title = {A global uniqueness theorem for an inverse boundary value problem},
  journal = {Annals of Mathematics},
  volume = {125},
  number = {1},
  pages = {153--169},
  year = {1987},
  doi = {10.2307/1971291}
}

@misc{uhlmann2026calderon,
  author = {Uhlmann, Gunther and Wang, Yiran},
  title = {The {Calder\'on} problem for near-{Euclidean} metrics},
  year = {2026},
  howpublished = {arXiv:2606.26540},
  note = {Preprint},
  eprint = {2606.26540},
  archivePrefix = {arXiv},
  primaryClass = {math.AP}
}

@article{ma2024anisotropic,
  author = {Ma, Shiqi and Sahoo, Suman K. and Salo, Mikko},
  title = {The anisotropic {Calder\'on} problem at large fixed frequency on manifolds with invertible ray transform},
  journal = {Journal of the London Mathematical Society},
  volume = {110},
  number = {4},
  pages = {e13006},
  year = {2024},
  doi = {10.1112/jlms.13006}
}

@article{ruland2019quantitative,
  author = {R{\"u}land, Angkana and Salo, Mikko},
  title = {Quantitative {Runge} approximation and inverse problems},
  journal = {International Mathematics Research Notices},
  volume = {2019},
  number = {20},
  pages = {6216--6234},
  year = {2019},
  doi = {10.1093/imrn/rnx301}
}

@article{uhlmann2024anisotropic,
  author = {Uhlmann, Gunther and Wang, Yiran},
  title = {The anisotropic {Calder\'on} problem for high fixed frequency},
  journal = {SIAM Journal on Mathematical Analysis},
  volume = {56},
  number = {3},
  pages = {4084--4103},
  year = {2024},
  doi = {10.1137/23M1579029}
}

@misc{uhlmann2026partial,
  author = {Uhlmann, Gunther and Wang, Yiran},
  title = {The partial data {Calder\'on} problem in dimension three},
  year = {2026},
  howpublished = {arXiv:2606.10247},
  note = {Preprint},
  eprint = {2606.10247},
  archivePrefix = {arXiv},
  primaryClass = {math.AP}
}

@book{adams2003sobolev,
  author = {Adams, Robert A. and Fournier, John J. F.},
  title = {Sobolev Spaces},
  edition = {Second},
  series = {Pure and Applied Mathematics},
  volume = {140},
  publisher = {Elsevier/Academic Press},
  address = {Amsterdam},
  year = {2003},
  url = {https://shop.elsevier.com/books/sobolev-spaces/adams/978-0-12-044143-3}
}

@article{alberti2019calderon,
  author = {Alberti, Giovanni S. and Santacesaria, Matteo},
  title = {{Calder\'on}'s inverse problem with a finite number of measurements},
  journal = {Forum of Mathematics, Sigma},
  volume = {7},
  pages = {e35},
  year = {2019},
  doi = {10.1017/fms.2019.31}
}

@article{angulo2016obstructions,
  author = {Angulo-Ardoy, Pablo and Faraco, Daniel and Guijarro, Luis and Ruiz, Alberto},
  title = {Obstructions to the existence of limiting {Carleman} weights},
  journal = {Analysis \& PDE},
  volume = {9},
  number = {3},
  pages = {575--595},
  year = {2016},
  doi = {10.2140/apde.2016.9.575}
}

@misc{chen2026quasianalyticity,
  author = {Chen, Lu and Jiang, Yan and Liu, Hongyu and Tao, Longyue},
  title = {Quasianalyticity and geometric rigidity in anisotropic {Calder\'on}'s problem},
  howpublished = {arXiv:2608.15813},
  note = {Preprint},
  year = {2026},
  eprint = {2608.15813},
  archivePrefix = {arXiv},
  primaryClass = {math.AP},
  url = {https://arxiv.org/abs/2608.15813}
}

@article{daude2019non,
  author = {Daud{\'e}, Thierry and Kamran, Niky and Nicoleau, Fran{\c{c}}ois},
  title = {Non-uniqueness results for the anisotropic {Calder\'on} problem with data measured on disjoint sets},
  journal = {Annales de l'Institut Fourier},
  volume = {69},
  number = {1},
  pages = {119--170},
  year = {2019},
  doi = {10.5802/aif.3240}
}

@article{daude2021anisotropic,
  author = {Daud{\'e}, Thierry and Kamran, Niky and Nicoleau, Fran{\c{c}}ois},
  title = {The anisotropic {Calder\'on} problem on 3-dimensional conformally {St{\"a}ckel} manifolds},
  journal = {Journal of Spectral Theory},
  volume = {11},
  number = {4},
  pages = {1669--1726},
  year = {2021},
  doi = {10.4171/JST/384}
}

@misc{daude2026stability,
  author = {Daud{\'e}, Thierry and Kamran, Niky and Nicoleau, Fran{\c{c}}ois},
  title = {Stability in the anisotropic {Calder\'on} problem for {Painlev\'e-Liouville} {Riemannian} manifolds},
  howpublished = {arXiv:2602.13703},
  note = {Preprint},
  year = {2026},
  eprint = {2602.13703},
  archivePrefix = {arXiv},
  primaryClass = {math.AP},
  url = {https://arxiv.org/abs/2602.13703}
}

@article{dos2013determining,
  author = {Dos Santos Ferreira, David and Kenig, Carlos E. and Salo, Mikko},
  title = {Determining an unbounded potential from {Cauchy} data in admissible geometries},
  journal = {Communications in Partial Differential Equations},
  volume = {38},
  number = {1},
  pages = {50--68},
  year = {2013},
  doi = {10.1080/03605302.2012.736911}
}

@article{guillarmou2011calderon,
  author = {Guillarmou, Colin and Tzou, Leo},
  title = {{Calder\'on} inverse problem with partial data on {Riemann} surfaces},
  journal = {Duke Mathematical Journal},
  volume = {158},
  number = {1},
  pages = {83--120},
  year = {2011},
  doi = {10.1215/00127094-1276310}
}

@article{kenig2011reconstructions,
  author = {Kenig, Carlos E. and Salo, Mikko and Uhlmann, Gunther},
  title = {Reconstructions from boundary measurements on admissible manifolds},
  journal = {Inverse Problems and Imaging},
  volume = {5},
  number = {4},
  pages = {859--877},
  year = {2011},
  doi = {10.3934/ipi.2011.5.859}
}

@book{taylor2011partial,
  author = {Taylor, Michael E.},
  title = {Partial Differential Equations {I}: Basic Theory},
  edition = {Second},
  series = {Applied Mathematical Sciences},
  volume = {115},
  publisher = {Springer},
  address = {New York},
  year = {2011},
  doi = {10.1007/978-1-4419-7055-8}
}
\end{document}